\documentclass{amsart}

\usepackage{amssymb}
\usepackage{amsmath,amsfonts,amsthm,amstext}
\usepackage{amsmath}
\usepackage{mathrsfs}  
\usepackage[all]{xy}
\usepackage{xcolor}
\usepackage{tikz-cd}

\newtheorem{theorem}{Theorem}[section]
\newtheorem{lemma}[theorem]{Lemma}

\newtheorem{prop}[theorem]{Proposition}
\newtheorem{cor}[theorem]{Corollary}

\newcommand{\Z}{\mathbb{Z}}
\newcommand{\z}{\mathbb{Z}}

\newcommand{\Q}{\mathbb{Q}}

\newcommand{\X}{{\mathcal X}}
\theoremstyle{definition}

\begin{document}

		\title[]{
		On the residual finiteness of the non-abelian  exterior square of some Artin groups,  constructible soluble groups, generalised Baumslag-Solitar groups  and $SL_n(\Z)$} \author{Lucas Barroso Rocha, Dessislava H. Kochloukova}
	\address{State University of Campinas (UNICAMP), SP, Brazil \\
		\newline
		email : lucasbarrosorocha@hotmail.com, desi@unicamp.br
		} 
	\email{}
	\date{}
	\keywords{}

\begin{abstract} We show that cohomological and homological goodness are equivalent. We prove that if $G$ is  a coherent even Artin group or an even Artin group whose  underlying graph does not contain 4-clique  or
 a LERF Artin group then    $G \otimes G$ and $G \wedge G$ are residually finite.  Furthermore for $G$ a RAAG both  $G \otimes G$ and $G \wedge G$ are residually $p$-finite for any prime integer $p$.
 For $G$ a limit group  or a Coxeter group  we show that  $G \wedge G$ is residually finite. 
 For $G$  a soluble group of type $FP_{\infty}$ we prove that the groups $\X(G)$, $\nu(G)$, $G \wedge G$ are virtually residually $p$-finite for almost all primes   $p$.
For $G$ a generalised Baumslag-Solitar group of rank $n \geq 1$ that is residually finite, we show that $G \wedge G$ is residually finite.
 For $G = SL_n(\Z)$ and $G = GL_n( \mathbb{Z})$, $n \geq 3$,  we show that $G \wedge G$ is residually finite. 
\end{abstract}

\maketitle

\section{Introduction}
 In this paper we continue the study of residually  finite  groups $G$ such that the non-abelian tensor product $G \otimes G$ and the non-abelian exterior product $G \wedge G$ are residually finite. We use homological   methods and make use of a  homological  version of the notion of good groups defined by  Serre in his book \cite{Serre}. Homological methods were quite often used in the early days of the study of the non-abelian tensor product $G \otimes G$, defined by  Dennis \cite{D} and developed by Brown and Loday \cite{B-L1,B-L2}. 
 Using homological methods Ellis showed in \cite{Ellis}  that if $G$ is finite then $G \otimes G$ is finite.
 In \cite{B-J-R}   Brown, Johnson and  Robertson proved  that for a free group $F$ of finite rank $n$ we have  $$F \otimes
F \simeq \mathbb{Z}^{n(n+1)/2} \times F',$$ hence $F \otimes F$ is residually finite.   In general there is a short exact sequence of groups
\begin{equation} \label{ses25}  1 \to \Delta(G) \to G \otimes G \to G \wedge G \to 1 \end{equation}
where $\Delta(G)$ is a central normal subgroup of $G \otimes G$ that is a quotient of the quadratic Whitehead functor $\Gamma(G^{ab})$ for $G^{ab}$ the abelianization of $G$. In the case $G = F$ we have $F \wedge F \simeq F'$ the commutator subgroup, so the short exact sequence (\ref{ses25}) splits. This argument was generalised in \cite{B-F-M} 
 where  Blyth, Fumagalli and Morigi  showed that if $G^{ab}$ is finitely generated and either $G^{ab}$ is 2-torsion free or $G'$ has complement in $G$ then $G \otimes G \simeq \Delta(G) \times ( G \wedge G)$.  It is important to note that by \cite{B-L1, B-L2} for a general group $G$ there is a short exact sequence $$1 \to H_2(G, \Z) \to G \wedge G \to G' \to 1$$  In particular, if $G$ is perfect $G \wedge G$ is isomorphic to the universal central extension of $G$.

 \subsection{Goodness} \ \ \ \ \ \ \ \ \ \ \ \ \ \ \ \ \ \medskip

 The notion of (cohomologically) good groups was developed by Serre in \cite{Serre}. A discrete group $G$ is called good if for the profinite completion $\widehat{G}$ of $G$ we have that the canonical map $G \to \widehat{G}$ induces an isomorphism \begin{equation} \label{introductionGood} H^i(\widehat{G}, A) \to H^i(G, A)\end{equation} for every finite discrete $\widehat{G}$-module $A$ and every $i \geq 0$.  Note that it will be clear from the context whether the (co)homology groups we use are in the category of profinite groups (are the so called countinuous ones) or in the category of discrete groups. Furthermore we will call good groups in Serre's sense (cohomologically) good.
 Following Serre \cite{Serre}  a group $G $ belongs to the class $\mathcal{A}_n$ if  (\ref{introductionGood}) holds for all $i \leq n$. In \cite{A-B-K}  a homological version $\mathcal{A}_n^{hom}$ of this property was considered:  a group $G $ belongs to the class $\mathcal{A}_n^{hom}$ if  the canonical map $G \to \widehat{G}$ induces an isomorphism $$H_i(G, A) \to H_i(\widehat{G}, A) $$ for every finite discrete  $\widehat{G}$-module $A$ and  $0 \leq i \leq n$. In \cite{A-B-K}  Anjos Teixeira, Barroso Rocha and Kochloukova  proved that if $G \in 
\mathcal{A}_2^{hom}$ and $H_2(G, \Z)$ is finitely generated then the map $$H_2(G, \Z) \to H_2( \widehat{G}, \widehat{\Z}) \hbox{ is injective }$$ and if $G$ is residually finite the injectivity of the above map implies that  $G \wedge G$ is residually finite.  It is important to note that the map $H_2(G, \Z) \to H_2( \widehat{G}, \widehat{\Z})$ factors through  the  map $\alpha:  H_2(G, \Z) \to \varprojlim_U H_2(G/U, \Z)$ and the injectivity of $\alpha$ together with the residual finiteness of $G$ is sufficient to deduce that $G \wedge G$ is residually finite.

  Our first result proves that equivalence between cohomological and homological goodness.
 
 \medskip
 {\bf Theorem A} {\it Let $G$ be a discrete group,  $\mathcal{T}$ be a set of normal subgroups of finite index in $G$ such that $\mathcal{T}$ defines the profinite topology of $G$ and for every prime integer $p$ consider $\mathbb{F}_p$ as the trivial $G$-module. Then the following are equivalent:

a) $G \in {\mathcal{A}}_n$;

b) $G \in {\mathcal{A}}_n^{hom}$;

c)  the map $H^i ( \widehat{U}, \mathbb{F}_p) \to H^i(U, \mathbb{F}_p)$ is an isomorphism for all $ i \leq n$, $U \in \mathcal{T}$ and for every prime integer $p$;

d)  the map $H_i(U, \mathbb{F}_p) \to H_i ( \widehat{U}, \mathbb{F}_p)$ is an isomorphism for all $ i \leq n$,  $U \in \mathcal{T}$ and  for every prime integer $p$.}
 
 \medskip Theorem A  generalises earlier results of Wilkes \cite[Prop. 2.2.2]{Wil} and Hillman and Kochloukova \cite{H-K}  that require the additional hypothesis that $G$ is of homological type $FP_{\infty}$. In section \ref{goodness} we establish a pro-p version of Theorem A.

\medskip

\subsection{Artin groups, graph product of groups, limit groups and Coxeter groups} \ \ \ \ \ \ \ \ \ \ \ \ \ \ \ \ \ \medskip

 Let $\Gamma$ be a  labelled finite simplicial graph, where  each edge has a label   an integer greater or equal to 2. The Artin group $G = G_{\Gamma}$ is generated by the set of vertices $V(\Gamma)$ and relations of the following type: for each edge $e$ in $\Gamma$ that links two vertices, say $a$ and $b$, we have the relation $aba \ldots = bab \ldots$ where both sides of the equality have length $m$ and $m$ is the label of  $e$. In the case when all labels are 2 the group is a right angled Artin group (RAAG). 
 In \cite{L2} Lorensen proved that a RAAG $G$ is residually $p$-finite i.e. the canonical map $G \to \widehat{G}_p$ is injective, for $\widehat{G}_p$ the pro-$p$ completion of $G$.  

\medskip

{\bf Theorem B} {\it Let $G$ be a RAAG. Then $G \wedge G$ and $G \otimes G$ are  residually $p$-finite for every prime $p$.}

\medskip
 In the next result we consider specific Artin groups.
Note that although many classes of Artin groups are residually finite in general not every Artin group is residually finite. In \cite{W-Y} Wu and Ye proved that the Artin group with a triangle graph and labels pairwise coprime integers has a simple commutator, in particular such group is not residually finite. In contrast Blasco-García,  Martínez-Pérez and Paris proved in \cite{BG-MP-L}  that even Artin groups of FC type are residually finite  and poly-free.  We will show in Theorem C that these groups are (cohomologically) good.
 Recall that an Artin group is called even if all labels of the corresponding graph are even.
In \cite{J-S}  Jankiewicz and Schreve showed that every even Artin group whose  underlying graph does not contain 4-clique is residually finite and (cohomologically) good.
We observe that by Theorem A we know that cohomological goodness and homological goodness coincide but we will continue referring to the common property as cohomological goodness, as it is known by this name.

A group is LERF if any finitely generated subgroup is closed in the profinite topology. The classification of  LERF Artin groups was completed by Almeida and Lima  in \cite{A-L}. The description  of LERF Artin groups $G$ in \cite{A-L}  implies that each connected component of the underlying graph of $G$ contains at most one edge with odd label.  Thus though these groups in general are not even Artin groups they are not too far away from being even.

  The class of limit groups over free groups was  studied by Sela and independently by Kharlampovich and  Miasnikov under the name of {fully residually free groups}, and it was quite  useful  in their work on solutions of systems of equations over free groups. Limit groups are residually free, in particular are residually finite. They are finitely presented and of type $FP_{\infty}$. In   \cite{G-J-Z} Grunewald,  Jaikin-Zapirain and Zalesskii showed that limits groups are (cohomologically) good.

By a result of Tits Coxeter groups are linear over $\mathbb{C}$, in particular they are  residually finite. Recently Hughes, M\"oller and  Verghese proved that Coxeter groups are (cohomologically) good \cite{H-M-V}.

\medskip
{\bf Theorem C} {\it Let $G$ be a    group of one of the following types:

a)  an even Artin group of FC type,

b)  an even Artin group whose  underlying graph does not contain 4-clique,

c) a LERF Artin group,

 d) a graph product of residually finite, (cohomologically) good groups,
 
  e) a limit group,

f) a Coxeter group.

Then $G$ is residually finite,  (cohomologically) good \iffalse{ of type $FP_{\infty}$}\fi  and $G \wedge G$ is residually finite. In cases a),b) and c)  $G \otimes G$ is residually finite.}

\medskip

\subsection{The Deligne example: $Sp_{2n}(\Z) \wedge Sp_{2n}(\Z)$ is not residually finite} \label{Deligne0}
 \ \ \ \ \ \ \ \ \ \ \ \ \ \ \ \ \  \ \ \ \ \ \ \ \ \ \ \ \ \ \ \ \ \ \ \  \ \ \ \ \ \ \ \ \ \ \ \ \ \ \ \ \ \ \ \

 \medskip
 In \cite{Deligne} Deligne showed that for $n \geq 2$ there is a central extension $$ 1 \to \Z \to \widetilde{ Sp_{2n}(\Z)} \to Sp_{2n}(\Z) \to 1$$ with $\widetilde{ Sp_{2n}(\Z)}$ not residually finite.  Actually Deligne worked with a subgroup of finite index in $ \widetilde{ Sp_{2n}(\Z)}$ and showed it is not residually finite. Assume now that $ n \geq 4$.
 In  \cite[Thm. 5.1]{Put} Putman showed that there is an isomorphism   $H_2(Sp_{2n}(\Z), \Z) \simeq \Z$, hence the above central extension is the universal central extension of  ${ Sp_{2n}(\Z)}$.
Furthermore the image of the central part $\Z$ in the profinite completion of $\widetilde{ Sp_{2n}(\Z)}$ is a  cyclic group of order 2.   Note that   since  $Sp_{2n}(\Z)$ is perfect the universal central extension is actually $ Sp_{2n}(\Z) \wedge  Sp_{2n}(\Z)$, so  $Sp_{2n}(\Z)$  is an example of a residually finite group whose non-abelian exterior square is not residually finite. 

We can view this using our methods: since $Sp_{2n}(\Z)$ has the congruence subgroup property and by the main result of Funar and Pitsch   $H_2(Sp_{2n}(\Z/ d \Z), \Z) \simeq \Z / 2 \Z$ if $4$ divides   $d$  \cite{Fu-Pi},
we get that  $\varprojlim_d H_2(Sp_{2n}(\Z/ d \Z), \Z)$ is either 0 or $\Z/ 2 \Z$, hence  the map  $\alpha:  \z \simeq  H_2(Sp_{2n}(\Z), \Z) \to \varprojlim_d H_2(Sp_{2n}(\Z/ d \Z), \Z)$ is not injective. Then the map $ H_2(Sp_{2n}(\Z), \Z) \to  H_2(\widehat{Sp_{2n}(\Z)}, \widehat{\Z})$ is not injective since it factors through $ \alpha$. By  \cite{A-B-K} if $G \wedge G$ is finitely generated and residually finite then the map $H_2(G, \Z) \to H_2(\widehat{G}, \widehat{\Z})$ is injective and we conclude that $Sp_{2n}(\Z) \wedge Sp_{2n}(\Z)$ is not residually finite.

 \subsection{$SL_n(\Z)$ and $GL_n(\Z)$} \ \ \ \ \ \ \ \ \ \ \ \ \ \ \ \ \ \medskip
 
The group $SL_n(\Z)$ is a finitely generated linear group over a commutative ring, hence it is residually finite. To handle the case of $SL_n(\Z)$  it is natural to use results from algebraic $K$-theory. 
As shown by Milnor in \cite{Milnorbook} for $ n \geq 5$ we have $H_2(SL_n(\Z), \Z) \simeq K_2(\Z) \simeq \Z / 2 \Z$ and similarly $H_2(SL_n(\Z/ s \Z), \Z) \simeq K_2(\Z/ s \Z)$. 
In \cite{Serre} Serre observed that $SL_n(\Z)$ is not (cohomologically) good for $n \geq 3$. Actually $SL_n(\Z) \not\in {\mathcal{A}}_2$ for $n \geq 3$  by the main result in \cite{Hewitt}. 
 We note that $G = SL_2(\Z)$ is virtually free, hence is (cohomologically) good and   in this case  $G \wedge G$ is residually finite.
For $G =   SL_n(\Z)$ we set $G_s =  SL_n(\Z/ s \Z)$.

\vskip 0.3cm
{\bf Theorem D} {\it  a) Let  $n \geq 3$ and  $G = SL_n(\Z)$. Then the map $H_2(G, \Z) \to \varprojlim_s H_2(G_s, \Z)$ is  an isomorphism and  $G \wedge G \simeq G \otimes G$  is residually finite. 

b) Let   $n \geq 3$  and $G = GL_n( \Z)$. Then  the map $H_2(G, \Z) \to H_2( \widehat{G}, \widehat{\Z}
)$ , induced by the canonical map $G \to \widehat{G}$, is  an isomorphism, $G \wedge G$  and $G \otimes G$ are   residually finite. }

\vskip 0.3cm

\subsection{ Soluble groups of type $FP_{\infty}$ and generalised Baumslag-Solitar groups} \ \ \ \ \ \ \ \ \ \ \ \ \ \ \ \ \ \medskip

    By a  result of Kropholler  \cite[Corollary to Thm. B]{Krop}  a group that is soluble of type $FP_{\infty}$ is constructible i.e.  the group can be built
up from the trivial group using finite extensions and HNN extensions. Furthermore by \cite{Strebel} such groups are linear over $\mathbb{Q}$, hence are residually finite by a well-known result of Malcev \cite{Mal}. An improvement to virtual residual $p$-finiteness for almost all prime integers $p$ was proved by Platonov in \cite{Plat}. We prove in Proposition \ref{Highly-res-fin} that   soluble groups of type $FP_{\infty}$ are (cohomologically) good in dimension $\leq 2$, i.e., are elements of $\mathcal{A}_2$, using results of Groves-Hermiller \cite{Gr-He} and Lorensen \cite{L}.

 We will consider generalised Baumslag-Solitar groups whose underlying graph is finite and  connected,  such groups are fundamental groups of finite graphs of groups
with vertex and edge groups isomorphic to $\mathbb{Z}^n$  such that all edge groups have finite index in the corresponding vertex groups. In \cite{Lev2} Levitt classified the residually finite groups among the generalised Baumslag-Solitar groups of rank $n= 1$. This was generalised to general rank $n$ in \cite{L-Sh} by showing that the residually finite ones could be only HNN extension of $\mathbb{Z}^n$ (hence soluble constructible) or virtually $\mathbb{Z}^n$-by-free (hence (cohomologically) good).
 In a technical paper \cite{C-Wy} Cohen and Wykowski  classified which generalised Baumslag-Solitar groups of rank $1$ are (cohomologically) good. Their classification  proves that residually finite generalised Baumslag-Solitar groups of rank $1$ are (cohomologically) good.  Whether this holds for general rank $n$ is an open question.

   The group of weak commutativity $\X(G)$ was defined by Sidki in \cite{Said} and  the group $\nu(G)$ was defined by Rocco in \cite{Norai}. The groups $\X(G)$ and $\nu(G)$ are strongly related and by  \cite[Prop. 2.6]{Norai} the non-abelian tensor-square
  $G \otimes G$ is isomorphic to the normal subgroup $[G, G^{\varphi}]$ of $\nu(G)$.  We will discuss in more details $\nu(G)$ in the preliminary section \ref{sect-prem-nu} and $\X(G)$  at the beginning of section \ref{coherence}.

  \medskip
 {\bf Theorem E}
{\it a) Let $G$ be a soluble group of type $FP_{\infty}$. Then  $G \in \mathcal{A}_2$ and for almost all prime integers  $p$ the groups $\X(G)$, $\nu(G)$, $G \wedge G$ are virtually residually $p$-finite. In particular,  $\X(G)$, $\nu(G)$, $G \wedge G$ are residually finite.

b) Let $G$ be a generalised Baumslag-Solitar group of rank $n \geq 1$ that is residually finite. Then $G \in \mathcal{A}_2$ and  $G \wedge G$ is residually finite.}
 
\medskip
\subsection{The non-abelian  $q$-exterior square} \ \ \ \ \ \ \ \ \ \ \ \ \ \ \ \medskip

The non-abelian  $q$-exterior square $G \wedge^q G$ was defined in  \cite{C-R} for an integer $q \geq 1$, where $G \wedge^1 G \simeq G$ and furthermore $G \wedge^0 G$ can be defined  as $ G \wedge G$.
In section \ref{q-section} we apply Corollary \ref{CorNew}  to deduce in Corollary \ref{q-square-res-finite} and Proposition \ref{q-sln-prop} that for the classes of groups $G$  studied in this paper  $G \wedge^q G$  is residually finite for $q \geq 2$.

\medskip
\subsection{Coherence} \ \  \ \ \ \ \ \ \ \medskip

 We know that the constructions $\nu(G)$ and $\X(G)$ preserve several properties  as $\pi$-finite \cite{Said}, polycyclic-by-finite \cite{LimaOliveira}, soluble (this follows directly by the structure theory developed in \cite{Said}), finitely generated nilpotent \cite{G-R-S}, finitely presented \cite{B-K}.
  In the final section \ref{coherence} of the paper we show that coherence is not preserved under the weak commutativity construction.

\medskip
{\bf Proposition G}
{\it   Let $F$ be a finitely generated  free group of rank {\color{red} $n \geq 2$}. Then 
  
  a) $\nu(F)$ is not coherent.
  
  b) $\X(F)$ is not coherent.}

  \medskip Note that for a free group $F$ we have $F \wedge F \simeq F'$ the commutator subgroup of $F$  and $F \otimes
F \simeq \mathbb{Z}^{n(n+1)/2} \times F'$, so $F \wedge F$  and $F \otimes F$ are definitely coherent in this case.

\subsection{Proofs location} \ \ \ \ \ \medskip

 We prove Theorem A in section \ref{goodness}, Theorem B in section \ref{sectionRAAG}, Theorem C follows from the results in section \ref{SectAr}, Theorem D  is proved in section \ref{SLnSection} and  Theorem E is considered in section  \ref{constr}.

\medskip
{\bf Acknowledgements} This study was financed in part by the Coordenação de Aperfeiçoamento de Pessoal de Nível Superior - Brasil (CAPES) - Finance Code 001 (Barroso Rocha is supported by a PhD Capes grant)  and Kochloukova is partially supported by bolsa de produtividade em pesquisa CNPq 301222/2026-6 and FAPESP 2024/14914-9. 

\section{Preliminaries}

\subsection{Goodness and residually finite groups}

   Recall that a discrete group $G$ is called {\it (cohomologically) good} if for the profinite completion $\widehat{G}$ of $G$ the canonical map $G \to \widehat{G}$ induces an isomorphism $$H^i(\widehat{G}, A) \to H^i(G, A)$$ for every finite discrete $\widehat{G}$-module $A$ and every $i \geq 0$.  The class of groups where  the isomorphism in cohomology is true in dimension $i \leq n$ we denote by $\mathcal{A}_n$.
 As stated in the introduction  a group $G $ belongs to the class $\mathcal{A}_n^{hom}$ if  the canonical map $G \to \widehat{G}$ induces an isomorphism $$H_i(G, A) \to H_i(\widehat{G}, A) $$ for every finite discrete  $\widehat{G}$-module $A$ and  $0 \leq i \leq n$. If this holds for all $i \geq 0$ we say that $G$ is {\it homologically good}. For properties of homology and cohomology of profinite groups we refer the reader to \cite{R-Z}.

Recall that $\widehat{G}$ denotes the profinite completion of a discrete group $G$.

\begin{prop} \cite{Serre} \label{sss}
If $G \in \mathcal{A}_2$ and $1 \to N \to E \to G \to 1$ is a short exact sequence of groups with $N$ finitely generated then there is a short exact sequence of groups $ 1 \to \widehat{N} \to \widehat{E} \to \widehat{G} \to 1$. 
\end{prop}

\begin{lemma} \cite{S} \label{S1} If $1 \to N \to E \to G \to 1$ is a short exact sequence of groups with $N, G \in \mathcal{A}_2$ and $N$ finitely generated then $E \in \mathcal{A}_2$.
\end{lemma}

 \begin{lemma} \cite{Serre} \label{Serre2}  Let $n \geq 2$ be an integer.

a) Let $H$ be a subgroup of finite index in $G$. Then $H \in \mathcal{A}_n$ if and only if 
 $G \in \mathcal{A}_n$;
 
 b)  If $1 \to N \to E \to G \to 1$ is a short exact sequence of groups with $N, G \in \mathcal{A}_n$, $N$ finitely generated, $H^q(N, M)$ finite for all $q \leq n$ and every finite $E$-module $M$ then $E  \in \mathcal{A}_n$.  
\end{lemma}

\begin{prop} \cite{L} \label{Lor1}
If $G \in \mathcal{A}_2$ and $1 \to N \to E \to G \to 1$ is a short exact sequence of groups with $N$ finitely generated, both $N$ and $G$ residually finite, then $E$ is residually finite. 
\end{prop}

A version of the above proposition with $G$ a good group can be found in \cite[Cor. 6.2]{G-J-Z}. Actually though the statement of \cite[Cor. 6.2]{G-J-Z} requires that $G$ is good the proof uses only that  $G \in \mathcal{A}_2$.

 It is known that a central extension of residually finite group is not necessary residually finite, see Deligne's example in \cite{Deligne} or section \ref{Deligne0}.

A discrete group $G$ is {\it (cohomologically) $p$-good in dimension $\leq n$} if for every finite pro-$p$ $\mathbb{Z}_p[[\widehat{G}_p]]$-module $A$ the canonical map $G \to \widehat{G}_p$ induces an isomorphism 
$$H^{i}(\widehat{G}_p, A) \to H^i(G, A) \hbox{ for every } i \leq n.$$
A discrete group $G$ is {\it homologically $p$-good in dimension $\leq n$} if 
for every finite pro-$p$ $\mathbb{Z}_p[[\widehat{G}_p]]$-module $A$ the canonical map $G \to \widehat{G}_p$ induces an isomorphism 
$$ H_i(G, A) \to H_{i}(\widehat{G}_p, A) \hbox{ for every } i \leq n.$$

 \subsection{Structure of $G \otimes G$ and $ G \wedge G$} \label{sect-prem-nu} The group $\nu(G)$ was defined in \cite{Norai}. Here $G^{\varphi}$ is a copy of the group $G$ with an isomorphism $G \to G^{\varphi}$ sending $g$ to $g^{\varphi}$. Then $\nu(G)$ has a presentation
 $$\langle G, G^{\varphi} \ | \ [g_1, g_2^{\varphi}]^{g_3} = [g_1, g_2^{\varphi}]^{g_3^{\varphi}} = [g_1^{g_3}, (g_2^{ g_3})^{\varphi}] \rangle$$
 By \cite[Prop. 2.6]{Norai} the non-abelian tensor square
  $G \otimes G$ is isomorphic to the normal subgroup $[G, G^{\varphi}]$ of $\nu(G)$, with an isomorphism that sends  $g_1 \otimes g_2$ to $[g_1, g_2^{\varphi}]$. 
 
 Let $W_0(G)$ be the kernel of the epimorphism  $$G \otimes G \to G'$$ that sends $g_1 \otimes g_2$ to $[g_1, g_2]$ and $\Delta(G)$ be the subgroup of $G \otimes G$ generated by $\{ g \otimes g ~ | ~ g \in G \}$. By definition  the non-abelian tensor square $G \wedge G$ is the quotient $G \otimes G/ \Delta(G)$ and $g_1 \wedge g_2$ is the image of $g_1 \otimes g_2$ in $G \wedge G$. 
 
 By \cite{B-L1}, \cite{B-L2} $\Delta(G)$ is a quotient of $\Gamma(G/ G')$, where $\Gamma$ denotes the Whitehead quadratic functor. Furthermore  $G \wedge G $ maps surjectively to $G'$, via the map that sends  $g_1 \wedge g_2$ to $[g_1, g_2] = g_1^{-1} g_2^{-1} g_1 g_2$,  with kernel isomorphic to the Schur multiplier $M(G) = H_2(G, \mathbb{Z})$ i.e. there is a short exact sequence of groups
 $$1 \to H_2(G, \mathbb{Z}) \to G \wedge G \to G' \to 1$$ In particular since $\Gamma(G/ G')$ is finitely generated if $G/ G'$ is finitely generated, we see that the group $\Delta(G)$ is finitely generated if $G/ G'$ is finitely generated. By \cite[p.68-69]{Norai} $\Delta(G)$ is  a central subgroup  in $\nu(G)$.

 \begin{theorem} \cite{B-F-M} \label{2tor}
 Let $G$ be a group.

 a)   Suppose $G/ G'$ is finitely generated. Furthermore $G/ G'$ has no
elements of order 2, or $G'$ has a complement in $G$, then $G \otimes G \simeq \Delta(G) \times (G \wedge G)$;

 b) Suppose $G \simeq F / R$, where $F$ is a free group then $G \wedge G \simeq F'/ [F,R]$. 
 \end{theorem}
 
 {\bf Remark} The proof of  Theorem \ref{2tor}, a) in \cite{B-F-M} shows that under the assumptions of Theorem \ref{2tor}, a) $\Delta(G) \simeq \Delta(G/ G') \simeq \Gamma(G/ G')$.

\medskip
    The non-abelian  $q$-exterior square $G \wedge^q G$ is defined in  \cite{C-R}
and by \cite{E-RF}  there is a  central extension \begin{equation} \label{ellis1}  1 \to H_2(G, \mathbb{F}_q) \to G \wedge^q G \to [G,G]G^q \to 1 \end{equation}
where $\mathbb{F}_q = \mathbb{Z}/ q \mathbb{Z}$. Thus for $q = 1$ we have $G \wedge^1 G \simeq G$ and for $q = 0$ we have $G \wedge^0 G \simeq G \wedge G$. 

\begin{lemma} \cite{Ellis2} \label{Wh-Wh}  Let $\Gamma$ be the Whitehead quadratic functor. Then there is an exact sequence
$$ \Gamma(G / [G,G] G^q) \to G \otimes^q G \to  G \wedge^q G \to 1$$
\end{lemma}

By Lemma \ref{Wh-Wh} 
 if $G$ is $q$-perfect, i.e. $G = [G,G]G^q$,  then $G \otimes^q G \simeq  G \wedge^q G$.
 \subsection{Known conditions for residual finiteness}

We denote by $\widehat{G}_p$ the  pro-$p$ completion of a discrete group $G$. Denote by $\mathbb{Z}_p$ the  pro-$p$ completion of $\mathbb{Z}$.  
 
\medskip

\begin{theorem}  \cite{A-B-K} \label{old1} Let $G$ be a discrete group.

   a1)  Suppose that $G$ is residually finite and the  map $$H_2(G, \mathbb{Z}) \to H_2(\widehat{G}, \widehat{\mathbb{Z}}),$$ induced by the canonical map  $G \to \widehat{G}$, is injective. Then $G \wedge G$ is residually finite.
 
 a2)  Suppose that $G$ is residually finite and the  map $H_2(G, \Z) \to \varprojlim_U H_2(G/U, \Z)$, induced by the canonical projections $G \to G/ U$ (for $ \widehat{G} \simeq \varprojlim G/U$, $G/ U$ finite groups) is injective. Then $G \wedge G$ is residually finite.

 b) 
  Suppose that $G$ is a finitely generated residually finite group such that $G \wedge G$ is finitely generated. Then $G \wedge G$ is residually finite if and only if the map  $$H_2(G, \mathbb{Z}) \to H_2(\widehat{G}, \widehat{\mathbb{Z}}),$$ induced by the canonical map  $G \to \widehat{G}$, is injective.
 
     c1) Suppose that $G$ is residually $p$-finite and the map $H_2(G, \Z) \to \varprojlim_U H_2(G/U, \Z)$, induced by the canonical projections $G \to G/ U$ (for $ \widehat{G}_p \simeq \varprojlim G/U$, $G/ U$ $p$-finite groups) is injective then $G \wedge G$ is residually $p$-finite.
 
 c2)  Suppose that $G$ is residually $p$-finite and    the canonical map $H_2(G, \mathbb{Z}) \to H_2(\widehat{G}_p, {\mathbb{Z}}_p)$, induced by the canonical map  $G \to \widehat{G}_p$, is injective then $G \wedge G$ is residually $p$-finite.

 d) If $G/ G'$ is finitely generated and   either has no 2-torsion or  $G'$ has a complement in $G$  then $G \otimes G$ is residually finite if and only if $G \wedge G$ is residually finite.
 
 e)   If $G/ G'$ is finitely generated and either has no 2-torsion or  $G'$ has a complement in $G$  then $G \otimes G$ is residually $p$-finite if and only if $G \wedge G$ is residually $p$-finite and the torsion subgroup of $\Delta(G)$ is $p$-finite. In particular  $G \otimes G$ is residually $p$-finite for every prime $p$ if and only if $\Delta(G)$ is torsion-free and $G \wedge G$ is residually $p$-finite for every prime $p$.

\end{theorem}

\begin{theorem} \cite{A-B-K} \label{equi} 
  If 
  $H_2(G, \Z/ m \Z) \to H_2( \widehat{G}, \Z/m \Z)$ is an isomophism for every integer $m \geq  2 $, where $\Z/ m \Z$ is the trivial $G$-module
 and $H_2(G, \mathbb{Z})$ is finitely generated then  $H_2(G, \mathbb{Z}) \to H_2(\widehat{G}, \widehat{\mathbb{Z}})$ is injective. In particular,  if      
  $G \in \mathcal{A}_2^{hom}$,  $H_2(G, \mathbb{Z})$ is finitely generated  and $G$ is residually finite then $G \wedge G$ is residually finite.
\end{theorem}

 \begin{prop} \cite{A-B-K} \label{iff25} Let $q \geq 2$ be an integer.
  Suppose that $G$ is a finitely generated residually finite group . Then $G \wedge^q G$ is residually finite if and only if the map $H_2(G, \mathbb{F}_q) \to H_2(\widehat{G}, \mathbb{F}_q)$, induced by the canonical map  $G \to \widehat{G}$, is injective.
  \end{prop} Note that by Theorem A $\mathcal{A}_2^{hom} = \mathcal{A}_2$. This together with Proposition \ref{iff25} implies the following corollary.
    
\begin{cor} \label{CorNew} 
Suppose that $G$ is a residually finite group such that $G \in \mathcal{A}_2$ and $q \geq 2$ is an integer.
Then $G \wedge^q G$ is residually finite.
\end{cor}

\subsection{The fundamental group of a graph of groups}  Let  $X$ be a finite graph with possible loops
 and $(\mathcal{G}, X) $ be a graph of groups, where $\mathcal{G} = \{ G_v, G_e \ | \ v \in V(X), e \in E(X) \}$ is a collection of groups attached to each vertex and edge of $X$ as described below. Let $A$ be an orientation of the set of edges $E(X)$ of $X$,   $o(e)$ and $\tau(e)$ denote the beginning and the end of the edge $e \in A$. For every vertex $v$ of the set of vertices $V(X)$ of  $X$, there is a group $G_v$ , for every edge $e \in A$ there is a
subgroup $G_e$ of $G_{o(e)}$ and an isomorphism $$\phi_e: G_e \to \phi_e(G_e) \leq G_{\tau(e)}$$

Let $T$ be a maximal forest of $X$. By definition the fundamental group $G$ of the graph of groups  $(\mathcal{G}, X) $ is denoted by
$\pi(\mathcal{G}, X) $ and is defined by a presentation (in terms of generators and relations) by
$$
 G = \pi(\mathcal{G}, X) =\langle \{ G_v \}_{v \in V(X)}, \{ t_e \}_{e \in A} \ | \  rel(G_v),$$ $$ t_e g_e t_e^{-1} = \phi_e(g_e) \hbox{ for } g_e \in G_e, t_e = 1, e \in A \cap E(T) \rangle$$

\section{Homological and cohomological goodness coincide} \label{goodness}

  Set $\mathbb{F}_p = \Z / p \Z$. The following is a corollary of \cite[Lemma 3.1]{G-J-Z}.

\begin{lemma} \label{Pavel} Let $G$ be a discrete group, $\mathcal{T}$ be a set of normal subgroups of finite index in $G$ such that $\mathcal{T}$ defines the profinite topology of $G$ and  for every prime integer $p$ and $U \in \mathcal{T}$ we have that the map $H^i ( \widehat{U}, \mathbb{F}_p) \to H^i(U, \mathbb{F}_p)$ is an isomorphism for all $ i \leq n$, where $\mathbb{F}_p$ is the trivial $U$-module. Then $G \in {\mathcal{A}}_n$.
\end{lemma}

\begin{proof} We induct on $n$, assume the result holds for $n-1$ and $\widehat{U} = \varprojlim_{U_0 \leq U, U_0 \in \mathcal{T}} U/ U_0$, hence $$H^n(\widehat{U}, \mathbb{F}_p) \simeq \varinjlim_{U_0} H^n(U/U_0, \mathbb{F}_p).$$ Hence the map $$\varinjlim_{U_0} H^n(U/U_0, \mathbb{F}_p) \to H^n(U, \mathbb{F}_p)$$ is an isomorphism, in particular is surjective.  In particular every element $[x] \not= [0] \in H^n(U, \mathbb{F}_p) $ is an image of some $[y] \not= [0] \in H^n(U/U_0, \mathbb{F}_p)$, hence the map $H^n(U, \mathbb{F}_p) \to H^n(U_0, \mathbb{F}_p)$ sends $[x]$ to $0$. Thus $\varinjlim_{U \in \mathcal{T}} H^n(U, \mathbb{F}_p) = 0$ and we can apply \cite[Lemma 3.1]{G-J-Z}.
\end{proof}

\begin{theorem} \label{surprise} Let $G$ be a discrete group,  $\mathcal{T}$ be a set of normal subgroups of finite index in $G$ such that $\mathcal{T}$ defines the profinite topology of $G$ and for every prime integer $p$ consider $\mathbb{F}_p$ as the trivial $G$-module. Then the following are equivalent:

a) $G \in {\mathcal{A}}_n$;

b) $G \in {\mathcal{A}}_n^{hom}$;

c)  the map $H^i ( \widehat{U}, \mathbb{F}_p) \to H^i(U, \mathbb{F}_p)$ is an isomorphism for all $ i \leq n$, $U \in \mathcal{T}$ and for every prime integer $p$;

d)  the map $H_i(U, \mathbb{F}_p) \to H_i ( \widehat{U}, \mathbb{F}_p)$ is an isomorphism for all $ i \leq n$,  $U \in \mathcal{T}$ and  for every prime integer $p$.
\end{theorem}

\begin{proof} 

Let $\mathcal{R}$ be a free resolution of the trivial $\Z G$-module $\Z$,  $M$ be a finite $p$-primary $G$-module and $\mathcal{C} = \mathcal{R} \otimes_{\Z G} M$. Using the standard adjoint isomorphism we have
$$Hom_{\z G} ( \mathcal{R}, Hom_{\Z} (M, \mathbb{Q}_p / \Z_p)) \simeq Hom_{\Z} ({\mathcal{R}} \otimes_{\Z G} M, \Q_p/ \Z_p)$$
where $\Z_p$ is the ring of $p$-adic integers, i.e. the pro-$p$ completion of $\Z$ and $\Q_p$ is the field of fractions of $\Z_p$. Note that $\Q_p/ \Z_p \simeq \Z[\frac{1}{p}]/ \Z$.

 Let $M^* = Hom_{\Z} (M, \Q_p/ \Z_p)$.
Hence
$$H^i(G, M^*) \simeq H^i(Hom_{\Z G} (\mathcal{R}, M^*)) \simeq H^i(Hom_{\Z}(\mathcal{C}, \Q_p/ \Z_p))$$  
Since $Hom_{\Z} ( - , \Q_p/ \Z_p)$ is an exact additive functor if considered as a functor on the category of $p$-primary $\z$-modules, it commutes with  homology, and  we have
$$
H^i(Hom_{\Z}(\mathcal{C}, \Q_p/ \Z_p)) \simeq Hom_{\Z}(H_i(\mathcal{C}),  \Q_p/ \Z_p) \simeq Hom_{\Z}(H_i(G, M),  \Q_p/ \Z_p)$$ 
This proves that
\begin{equation} \label{sur1} H^i(G, M^*) \simeq H_i(G, M)^* \end{equation}
A profinite version is valid too
\begin{equation} \label{sur2} H^i(\widehat{G}, M^*) \simeq H_i(\widehat{G}, M)^* \end{equation}
which is known as Pontryagin duality in profinite groups \cite[Prop. 6.3.6]{R-Z}. It could be obtained from (\ref{sur1}) applied for $G/U$ and then taking direct limit over all normal subgroups  $U$ of finite index in $G$.

a) implies b) We have to show that for every finite $G$-module $M$ the map $H_i(G, M) \to H_i ( \widehat{G}, M)$ is an isomorphism for $i \leq n$. Since $M$ is a direct sum of  its $p$-primary components we can assume that $M$ is $p$-primary. Since $G \in {\mathcal{A}}_n$ then  $ H^i(\widehat{G}, M^*) \to H^i(G, M^*) $ is an isomorphism, hence by (\ref{sur1}),(\ref{sur2}) we get that the map $$\beta: H_i(\widehat{G}, M)^* \to 
H_i(G, M)^*,$$ induced by $G \to \widehat{G}$, is an isomorphism. Note that $\beta = \alpha^*$ where 
$$\alpha: H_i(G, M) \to H_i ( \widehat{G}, M)$$ is induced by the canonical map $G \to G^*$, $\beta$ is surjective if and only if $\alpha$ is injective and $\beta$ is injective if and only if $\alpha$ is surjective. Thus $\alpha$ is an isomorphism.

b) implies a) We have to show that for every finite $G$-module $M$ the map $H^i(\widehat{G}, M) \to H^i ( {G}, M)$ is an isomorphism for all $ i \leq n$ . Again it suffices to consider the case when $M$ is $p$-primary.  By (\ref{sur1}), (\ref{sur2}) together with $M^{**} \simeq M$ (this holds since $M$ is finite $p$-primary)  the map $H^i(\widehat{G}, M) \to H^i ( {G}, M)$  is an isomorphism if and only if 
the map $$ H_i(\widehat{G}, M^*)^* \to H_i(G, M^*)^*,$$ induced by the canonical map $G \to \widehat{G}$, is an isomorphism. But we know that the map $H_i(G, M^*) \to H_i( \widehat{G}, M^*)$ is an isomorphism, hence its dual is an isomorphism.

a) equivalent to c) This follows by Lemma \ref{Pavel}   (for c) implies a)) and the Shapiro Lemma in cohomology (for a) implies c)). 

b) implies d) This  follows from the Shapiro Lemma in homology.

d) implies c) Just substitute $G$ with $U$ in the proof of b) implies a) and set $M = \mathbb{F}_p$.
\end{proof}

As a corollary of the above result that implies the equivalence of $\mathcal{A}_2$ and $\mathcal{A}_2^{hom}$ and Theorem \ref{equi}  we obtain:

 \begin{cor} \label{equi2}
If $G \in \mathcal{A}_2$, $G$ is residually finite and $H_2(G, \z)$ is finitely generated then the map $H_2(G, \Z) \to H_2( \widehat{G}, \widehat{\Z})$ is injective and $G \wedge G$ is residually finite.
\end{cor}

  We say that $G$ is {\it  (cohomologically) $p$-good in dimension $\leq n$} if for every finite pro-$p$ $\mathbb{Z}_p[[\widehat{G}_p]]$-module $A$ the canonical map $G \to \widehat{G}_p$ induces an isomorphism 
$$H^{i}(\widehat{G}_p, A) \to H^i(G, A) \hbox{ for every } i \leq n,$$
where $H^{i}(\widehat{G}_p, A)$ denotes continuous cohomology. 
We say $G$ is {\it homologically $p$-good in dimension $\leq n$} if 
for every finite pro-$p$ $\mathbb{Z}_p[[\widehat{G}_p]]$-module $A$ the canonical map $G \to \widehat{G}_p$ induces an isomorphism 
$$ H_i(G, A) \to H_{i}(\widehat{G}_p, A) \hbox{ for every } i \leq n,$$
where $H_i(\widehat{G}_p, A)$ denotes profinite (continuous) homology.

 \begin{theorem} \label{surprise-p0} Let  $p$ be a prime integer  and $G$ be a discrete group. Then the following are equivalent:

a) $G$ (cohomologically) $p$-good in dimension $\leq n$;

b) $G$ homologically $p$-good in dimension $\leq n$.

\end{theorem}

\begin{proof} 
The proof that a) and b) are equivalent is the same as the proof of the equivalence of  itens a) and b) from Theorem \ref{surprise} where we substitute $\widehat{G}$ with $\widehat{G}_p$. 
\end{proof}

\begin{prop} \label{surprise-p} Let $p$ be a prime integer and $G$ be a discrete group.

a) if the map $H^i ( \widehat{G}_p, \mathbb{F}_p) \to H^i(G, \mathbb{F}_p)$ is an isomorphism for all $ i \geq 1$ then $G$ is (cohomologically) $p$-good in all dimensions;

b) if the map $H_i(G, \mathbb{F}_p) \to H_i ( \widehat{G}_p, \mathbb{F}_p)$ is an isomorphism for all $ i \geq 1$ then $G$ is homologically $p$-good in all dimensions. \end{prop}

\begin{proof} We prove b). The proof of a) is similar. 
 Consider a  finite pro-$p$ $\mathbb{Z}_p[[\widehat{G}_p]]$-module $M$. We show by induction on $|M|$ that $H_i(G, M) \to H_i ( \widehat{G}_p, M)$ is an isomorphism for all $ i \geq 1$.
Since every finite pro-$p$ $\mathbb{Z}_p[[\widehat{G}_p]]$-module $M$ has a filtration,  where each quotient is a simple module and the trivial module $\mathbb{F}_p$ is the unique simple pro-$p$ $\mathbb{Z}_p[[\widehat{G}_p]]$-module we get that the basis of the induction holds for $M = \mathbb{F}_p$. If $|M| > p$ then there is a short exact sequence $0 \to M_1 \to M \to M_2 \to 0$ of pro-$p$ $\mathbb{Z}_p[[\widehat{G}_p]]$-modules and we have the following commutative diagram where the horizontal maps are the corresponding long exact sequences in homology and the vertical maps are induced by the canonical map $G \to \widehat{G}_p$
  \[ 
\begin{tikzcd}
  H_{i+1}(G, M_2) \arrow{r}{}
\arrow{d}{\varphi_{i+1,2}}
 &     H_i(G, M_1) \arrow{r}{}
\arrow{d}{\varphi_{i,1}}
 &    H_i(G,M) \arrow{r}{} \arrow{d}{\varphi_{i}}&  H_{i}(G,M_2) \arrow{r}{} \arrow{d}{\varphi_{i,2}}&  H_{i-1}(G,M_1)  \arrow{d}{\varphi_{i-1,1}}\\%
 H_{i+1}(\widehat{G}_p, M_2) \arrow{r}{}  &   H_i(\widehat{G}_p, M_1) \arrow{r}{}  &    H_i(\widehat{G}_p, M) \arrow{r}{} & H_{i}( \widehat{G}_p, M_2) \arrow{r}{} & H_{i-1}(\widehat{G}_p, M_1)\\%
 \end{tikzcd}
\]
By the inductive hypothesis the vertical maps  $\varphi_{i+1,2}$, $\varphi_{i,1}$, $\varphi_{i,2}$ and $\varphi_{i-1,1}$ are isomorphisms for $i \geq 1$, hence $\varphi_{i}$ is an isomorphism  for $i\geq 1$. 
\end{proof}

\section{Right angled Artin groups} \label{sectionRAAG}

We start with a generalization of Theorem \ref{equi}.

  \begin{lemma} \label{inj} a) Suppose that $G$ is a discrete group with $H_i(G, \mathbb{Z})$ finitely generated  and   $H_i(G, \mathbb{Z}/ n \mathbb{Z}) \to H_i(\widehat{G}, \mathbb{Z}/ n \mathbb{Z})$ is an isomorphism for  every integer $n \geq 2$  then $H_i(G, \mathbb{Z}) \to H_i(\widehat{G}, \widehat{\mathbb{Z}})$ is injective. In particular if $G \in \mathcal{A}_i^{hom}$ and $H_s(G, \mathbb{Z})$ is finitely generated for every $s \leq i$ then $H_s(G, \mathbb{Z}) \to H_s(\widehat{G}, \widehat{\mathbb{Z}})$ is injective for every $s \leq i$. 
  
  b) If $G$ is  a discrete group with $H_i(G, \mathbb{Z})$ finitely generated  and  $H_i(G, \mathbb{Z}/ p^j \mathbb{Z}) \to H_i(\widehat{G}_p, \mathbb{Z}/ p^j \mathbb{Z})$ is an isomorphism   for a fixed prime $p$ and every integer $j \geq 1$ then $Ker(H_i(G, \mathbb{Z}) \to H_i(\widehat{G}_p, {\mathbb{Z}_p}))$ is a   $p'$-finite group i.e. has order coprime to $p$.
  \end{lemma}
  
  \begin{proof}   a)    By the universal coefficients theorem 
   there is an exact sequence
    $$0 \to H_i(G, \mathbb{Z}) \otimes_{\mathbb{Z}} \mathbb{Z}/ n \mathbb{Z} \to H_i(G,  \mathbb{Z}/ n \mathbb{Z}) $$ Since inverse limit is left exact functor applying inverse limit over all positive integers $n$ we get an exact sequence
    \begin{equation} \label{eq111} 0 \to \varprojlim_n H_i(G, \mathbb{Z}) \otimes_{\mathbb{Z}} \mathbb{Z}/ n \mathbb{Z} \to \varprojlim_n H_i(G,  \mathbb{Z}/ n \mathbb{Z}) \end{equation}
    By assumption  there is an isomorphism $H_i(G,  \mathbb{Z}/ n \mathbb{Z}) \simeq H_i(\widehat{G},  \mathbb{Z}/ n \mathbb{Z})$, hence
    \begin{equation} \label{eq222}  \varprojlim_n H_i(G,  \mathbb{Z}/ n \mathbb{Z}) \simeq \varprojlim_n H_i(\widehat{G},  \mathbb{Z}/ n \mathbb{Z}) \simeq H_i(\widehat{G}, \varprojlim_n \mathbb{Z}/ n \mathbb{Z}) \simeq H_i(\widehat{G}, \widehat{\mathbb{Z}}). 
    \end{equation}
    We observe  that  the map
   $$ H_i(G, \mathbb{Z}) \to \varprojlim_n H_i(G, \mathbb{Z}) \otimes_{\mathbb{Z}} \mathbb{Z}/ n \mathbb{Z}$$ has kernel $ \cap_n n H_i(G, \mathbb{Z})=0$, the last
   holds since $H_i(G, \mathbb{Z}) $ is finitely generated abelian. 
   Then the proof is completed by (\ref{eq111}) and (\ref{eq222}).

b)  We can repeat part of the argument of a) with $\widehat{G}$ substituted by $ \widehat{G}_p$ and $n =p^j$ to deduce that $$Ker(H_i(G, \mathbb{Z})  \to  H_i( \widehat{G}_p, \mathbb{Z}_p)) \subseteq  \cap_{j \geq 1} p^j H_i(G, \mathbb{Z})$$ 
   Then since $H_i(G, \mathbb{Z})$ is a finitely generated abelian group $ \cap_{j \geq 1} p^j H_i(G, \mathbb{Z})$ is a   $p'$-finite group i.e. has order coprime to $p$.  \end{proof}
  
Lemma \ref{inj}, b)  and Theorem \ref{old1}, c2) imply the following corollary.

\begin{cor}  \label{cor-1} Let $p$ be a prime integer. Suppose that $G$ is  a discrete group with $H_2(G, \mathbb{Z})$ finitely generated and  $H_2(G, \mathbb{Z}/ p^j \mathbb{Z}) \to H_2(\widehat{G}_p, \mathbb{Z}/ p^j \mathbb{Z})$ is an isomorphism for every integer $j \geq 1$ and the torsion part of $H_2(G, \mathbb{Z})$ is a finite $p$-group.  Then $H_2(G, \mathbb{Z}) \to H_2(\widehat{G}_p, {\mathbb{Z}_p})$ is injective map. If additionally $G$ is residually $p$-finite then $G \wedge G$ is residually $p$-finite.
\end{cor}

\begin{cor} Let $G$ be a RAAG. Then the canonical map $H_i(G, \mathbb{Z}) \to H_i(\widehat{G}_p, {\mathbb{Z}_p})$ is injective. In particular $G \wedge G$ and $G \otimes G$ are residually $p$-finite.
\end{cor} 

\begin{proof}

In \cite{L2} it was proved that for a RAAG $G$ the map $$H^i(\widehat{G}_p, \mathbb{F}_p) \to H^i(G, \mathbb{F}_p)$$ is an isomorphism  for every $i$,  where $\mathbb{F}_p$ is the trivial $\widehat{G}_p$-module. Then we can apply Theorem \ref{surprise-p}, Theorem \ref{surprise-p0} and conclude that $G$ is homologically $p$-good in all dimensions.
 Furthermore in \cite{L2} it is proved that $G$ is residually $p$-finite.  The fact that RAAGs are  (cohomologically) good was proved in \cite{L} but we will not use it here.

We claim that $H_2(G, \Z)$ is torsion-free. Let $\mathcal{P}$ be the free resolution of the trivial $\z G$-module $\z$ that  comes from the Salveti complex of $G$. Then $\mathcal P \otimes_{\Z G} \Z$ has trivial differentials. Hence $$H_2(G, \Z) \simeq H_2(\mathcal P \otimes_{\Z G} \Z) \simeq \Z E  \hbox{ is free abelian,}$$ where $E$ is the set of edges of the simplicial graph that defines $G$.   Then Corollary \ref{cor-1} implies that
$H_2(G, \mathbb{Z}) \to H_2(\widehat{G}_p, {\mathbb{Z}_p})$ is injective  and  $G \wedge G$ is residually $p$-finite.

Finally since $G^{ab}$ does not have 2-torsion we have by   Theorem \ref{2tor} that $G \otimes G = \Delta(G) \times ( G \wedge G)$. Note that by \cite[Prop. 2.2 iii]{B-F-M} we have that $$\Delta(G) = \Delta(G^{ab}) = Ker (G^{ab} \otimes G^{ab} \to G^{ab} \wedge G^{ab})$$
Since $G^{ab}$ is finitely generated, torsion-free we conclude that $G^{ab} \otimes G^{ab}$ and any subgroup of  $G^{ab} \otimes G^{ab}$ is finitely generated, torsion-free, in particular $\Delta(G)$ is finitely generated, torsion-free, hence  is residually $p$-finite.

\end{proof} 

  \section{More Artin groups} \label{SectAr}
  
  Let $\Gamma$ be a finite labelled simplicial graph, i.e., no loops and no more than one edge could link two vertices if any,  where each edge has a label  that is an integer greater or equal to 2. Denote by $V(\Gamma)$ the set of vertices of $\Gamma$. Recall that the Artin group $$ G_{\Gamma} = \langle V(\Gamma)  \ | \ w(a,b) = v( a,b)  \hbox{ where } a,b \hbox{ are linked by an edge in } \Gamma \rangle$$
  where $w(a,b) = aba \ldots$ and $v(a,b) = bab \ldots $ are  both alternating products of $a$ and $b$ of length  that equal the label of the edge that links $a$ and $b$ in $\Gamma$.

 A group $G$ is of homological type $FP_{\infty}$ if there is a projective resolution of the trivial $\Z G$-module with all projectives finitely generated.
  \begin{lemma} \label{dihedral}  2-generated Artin groups $G$, i.e. dihedral Artin groups,  are (cohomologically) good, residually finite and of type $FP_{\infty}$. \end{lemma} 

\begin{proof}
1)If the Artin group corresponds to $m = 2s$ then $G = \langle a,b \ | \ (ab)^s = (ba)^s \rangle$. Then for $x = ab$ we have $x^s = b x^s b^{-1}$ and  the normal subgroup $K$ generated as normal subgroup by $x^s$ and $b$ has finite index in $G$ and $K = \langle x^s, b, b^x, b^{x^2}, \ldots, b^{x^{s-1}} \rangle \simeq A \times B$, where $A$ is the infinite cyclic group generated by $x^s$ and $B$ is the free group with a free basis $b, b^x, \ldots, b^{x^{s-1}}$. Then $K$ is a direct product of (cohomologically) good groups, so is (cohomologically) good. Furthermore $K$ is of type $FP_{\infty}$. By Lemma \ref{Serre2} a) any finite extension of $K$ is  (cohomologically)  good. Note that $K$ and $G$ are residually finite.

2)If $m = 2s+ 1$ then $(ab)^s a = b(ab)^s$. Set $x = ab$, so $x^s x b^{-1} = b x^s$, so $G = \langle x, b \rangle$. Note that $G^{ab} \simeq \Z$ with $\bar{x} = \bar{b}^2$. Set $u = x b^{-2}$, then the relation $x^{s+1} b^{-1} x^{-s} b^{-1} = 1$ can be written as 
$1 =(u b^2)^{s+1} b^{-1} (u b^2)^{-s} b^{-1} =u_0  u_{-2} u_{-4} \ldots u_{-2s} u_{-(2s-1)}^{-1} u_{-(2s-3)}^{-1}$ $ \ldots u_{-1}^{-1} $,
where $u_i = b^{-i} u b^i$. Conjugating with an appropriate power of $b$ we get
$$u_{2s} u_{2s-2} u_{2s-4} \ldots u_0 u_1^{-1} u_3^{-1} \ldots  u_{2s-1}^{-1} = 1$$
This implies that the normal subgroup $N$ generated as a normal subgroup by $u$ is $N = \langle u_0, u_1, \ldots, u_{2s-1} \rangle$ i.e. is finitely generated. Note that as a group $N$ has an infinite presentation $$N = \langle \{ u_i \}_{i \in \Z} \ | \ u_{2s+ i} u_{2s-2+i} \ldots u_i u_{1+i}^{-1} u_{3+i}^{-1} \ldots  u_{2s-1+i}^{-1}  = 1 , i \in \Z \rangle$$ that reduces to
$N = \iffalse{\langle u_0, \ldots, u_{2s} \ | \ u_{2s} u_{2s-2} \ldots u_0 u_1^{-1} u_3^{-1} \ldots  u_{2s-1}^{-1} = 1 \rangle} \simeq}\fi \langle u_0, \ldots, u_{2s- 1} \ | \ \emptyset \rangle$
i.e. $N$ is a free group. Thus $G$ is (finitely generated free)-by-(infinite cyclic) group.
Hence by Lemma \ref{Serre2} b) $G$ is  (cohomologically)  good. By Proposition \ref{Lor1} $G$ is residually finite.

\end{proof}

  \begin{lemma}
  Let $G$ be a poly-(finitely generated free) group. Then $G$ is (cohomologically) good, residually finite and $G \wedge G$ is residually finite.
  \end{lemma}
 
  \begin{proof}
  Every  finitely generated free group is (cohomologically) good.
  An extension of (cohomologically) good groups of type $FP_{\infty}$ is again a  (cohomologically) good group by Theorem \ref{Serre2},b) and is of type $FP_{\infty}$ \cite{Bieribook}. By Proposition  \ref{Lor1}  $G$ is residually finite.  Since  $G$ is of type $FP_{\infty}$, we have $H_2(G, \Z)$ is finitely generated,    hence  by Corollary \ref{equi2} $G \wedge G$ is residually finite.
  \end{proof}

    \begin{lemma}
a)   Let $B_n$ be the  braid group on $n$ strings  and $P_n$ be the pure braid group on $n$ strings i.e. $B_n$ is an aspherical Artin group of type $A$. Then both $P_n$, $B_n$ are (cohomologically) good, in particular $B_n \wedge B_n$ and $P_n \wedge P_n$ are residually finite.  Furthermore $B_n \otimes B_n$ and $P_n \otimes P_n$ are residually finite.

b) Let $M$ be a compact, connected surface without boundary, and different
from $S^2$ and $RP^2$ and  $G = P_n(M )$. Then $G$ is (cohomologically) good, $FP_{\infty}$ and residually finite. In particular $G \wedge G$ is residually finite.
  \end{lemma}
  
  \begin{proof}
a) By \cite{Artin} $P_n$ can be decomposed as itterated split extension of finitely generated free groups i.e. $P_n \simeq F_{n-1} \rtimes P_{n-1}$ with $P_2 = \Z$. By the previous lemma $P_n$ is (cohomologically) good, residually finite  and $P_n \wedge P_n$ is residually finite. The fact that $B_n$ is (cohomologically) good was earlier proved in \cite{G-J-Z}.  Note that $P_n$ has finite index in $B_n$, hence by   Lemma \ref{Serre2},b) we get a new proof that   $B_n$  is (cohomologically) good and   by Lemma \ref{Lor1} $B_n$ is residually finite. Alternatively $B_n$ is spherical ( finite type) Artin group, hence $B_n$ is  a linear group, hence residually finite. Thus  by Corollary \ref{equi2} $B_n \wedge B_n$ is residually finite.

We note that by the first chapter of \cite{Ka-Tu} we have presentations of $B_n$ and $P_n$ in terms of generators and relations that show that the abelianizations of both groups are finitely generated, torsion-free. Hence  by Theorem \ref{old1}, d)  $B_n \otimes B_n$ and $P_n \otimes P_n$ are residually finite.

b)
By  \cite[Thm. 5]{Go-Gu} $G$ is isomorphic to a repeated semi-direct product of (fin. generated)  free groups
and a single one-relator group that is isomorphic to $\pi_1(M)$. Note that $\pi_1(M)$ is a surface group, hence is  $FP_{\infty}$ and a limit group over free group, hence by \cite{G-J-Z}  is (cohomologically) good. As in the previous results using that finitely generated free groups are (cohomologically) good and $FP_{\infty}$ we conclude that $G$ is (cohomologically) good and $FP_{\infty}$
  and residually finite.  Then we can apply Corollary \ref{equi2}.
\end{proof}

 Following \cite{G-J-Z} if $G$ is the fundamental group of a finite graph of groups the profinite topology of $G$ is efficient if 
 
  1) $G$ is residually finite;
 
 2) the restriction of the profinite topology  of $G$  on each vertex $G_v$   is the full profinite topology of $G_v$ and for each edge group $G_e$,  where $o(e)$ is the beginning and $\tau(e)$ the end of $e$,  considering the structure injective maps the inclusion $i_e: G_e \to G_{o(e)}$ and $\phi_e : G_e \to G_{\tau(e)}$, the restriction of the profinite topology of $G_{o(e)}$ on $i_e(G_e)$ and  the restriction of the profinite topology of $G_{\tau(e)}$ on $\phi_e(G_e)$ is the full profinite topology of $i_e(G_e)$ and $\phi_e(G_e)$ respectively.

3) all vertex groups $G_v$ and edge groups $G_e$ are separable (=closed in the profinite topology i.e. intersection of finite index subgroups of $G$).

\begin{prop} \label{good-short} \cite[Prop. 3.6]{G-J-Z}
Let $G$ be the fundamental group of a finite graph of groups with one edge (i.e. profinite amalgamated product or profinite HNN-extension) such that  the profinite topology of $G$ is efficient. Then if all vertex and edge groups are  (cohomologically) good   then $G$ is (cohomologically) good.
\end{prop} 

  Let $\Gamma$ be a simplicial finite graph and  for each vertex $v \in V(\Gamma)$ there is a group $G_v$. The  associated graph product of groups  is a group that  has the presentation
$$G = \langle \cup_v G_v \ | \ [G_v, G_w] = 1 \hbox{ if v,w are linked by an edge in } \Gamma \rangle $$
By \cite{Green} 

1) if each $G_v$ is residually finite then $G$ is residually finite, 

 2) each full subgroup of $G$ ( i.e. subgroup generated by some of the vertex groups) is separable in $G$ (i.e. closed in the profinite topology of $G$)

Alternative proofs of both results described above can be found in \cite{Hsu-Wise}. 

Note that by \cite{B-E} free product with amalgamation $G = G_1 *_{G_0} G_2$, where the amalgamated subgroup $G_0$ is a retract in both subgroups $G_1$ and $G_2$  and $G_1, G_2$ are residually finite, is  residually finite. 

\begin{prop}   Let $\Gamma$ be a simplicial finite graph and  $G$ be the graph product of the groups $G_v$, where $v \in V(\Gamma)$.
 If  $G_v \in \mathcal{A}_n$ for every $v \in V(\Gamma)$ then $G \in \mathcal{A}_n$.
\end{prop}
 
 \begin{proof} 
If $\Gamma$ is the complete graph on $V(\Gamma)$ then $G = \prod_v G_v$ and the result is  obvious. We use induction on the number of vertices in the graph.

If $\Gamma$ is not the complete graph then there are vertices $v$ and $w$ that are not linked by an edge. Let $G_1$, $G_2$ and $G_0$ be the graph products   based on the full subgraphs  of $\Gamma$  spanned by $V(\Gamma) \setminus \{ v \}$, $V(\Gamma) \setminus \{ w \}$ and $V(\Gamma) \setminus \{ v, w \}$ respectively. Then $G = G_1 *_{G_0} G_2$. Note that  $G_0$ is a retract of $G_1$ and of $G_2$ and $G_1$ and $G_2$ are retracts of $G$. 
Hence the profinite topology of $G$ induces the profinite topology of each of the groups $G_i$ for $i = 0,1,2$.   We already commented that by  \cite{Green},\cite{Hsu-Wise} $G$ is residually finite and full subgroups are separable, in particular $G_i$ for $i = 0,1,2$ are separable subgroups of $G$. Then the decomposition $G = G_1 *_{G_0} G_2$ is efficient and we can apply Proposition \ref{good-short}. 
\end{proof}

We recall that a group is LERF if any finitely generated subgroup is closed in the profinite topology. 
LERF Artin  groups are  classified in  \cite{A-L} as follows.
A group $G$ is a LERF Artin group if and only if $G \in A(S)$ where $A(S)$ is the smallest class of groups such that:

1) all Artin groups with maximum 2 generators belongs to $A(S)$;

2) if $G_1, G_2 \in A(S)$ then $G_1 * G_2 \in A(S)$;

3) if $G \in A(S)$ then $G \times \mathbb{Z} \in A(S)$.

\begin{prop} Let $G$ be an Artin group that is LERF. Then

a) $G$ is  residually finite  and of type $FP_{\infty}$,

b) $G$ is (cohomologically) good and finitely presented, in particular  $H_2(G, \Z) \to H_2( \widehat{G}, \widehat{\mathbb{Z}})$ is injective,

c) $G \wedge G$   and $G \otimes G$ are residually finite.
\end{prop}

\begin{proof}
a) If an Artin group has at most two generators then by Lemma \ref{dihedral}   it is residually finite  and is of type $FP_{\infty}$. If $G$ is residually finite then $G \times \mathbb{Z}$ is residually finite  and is of type $FP_{\infty}$.  If $G_1$ and $G_2$ are residually finite and of type $FP_{\infty}$  then $G_1 * G_2$ is residually finite  and is of type $FP_{\infty}$.

b)  Note that by \cite[Lemma 4.8]{Reid} the decomposition $G_1 * G_2$ of residually finite groups $G_1$ and $G_2$ is efficient. By Proposition \ref{good-short}   if $G_1$ and $G_2$ are  (cohomologically)   good  then $G_1 * G_2$ is (cohomologically) good.
Furthermore if $G$ is finitely generated  and  (cohomologically) good, then $G \times \Z$ is  (cohomologically)   good. Since $G$ is finitely presented (or by using  that $G$ is of type $FP_{\infty}$), then its second homology with integral coefficients is finitely generated and we can apply Corollary \ref{equi2}.

c) follows from  Corollary \ref{equi2} and the fact that $G^{ab}$ is torsion-free.
\end{proof}

 Recall that an Artin group is called even if all labels of the corresponding graph are even. We note that if $G = G_{\Gamma}$ is an even Artin group and $\Gamma_1$ is a subgraph of $\Gamma$ then there is an epimorphism of groups $G \to G_{\Gamma_1}$ that is the identity on the vertices in $\Gamma_1$ and sends all vertices outside $\Gamma_1$ to the identity.  Thus $G_{\Gamma_1}$ is a retract of $G$.

Recall that a 4-clique is a subgraph of a graph consisting of four vertices, where every vertex is connected to every other vertex by an edge.

\begin{prop}
a) Let $\Gamma$ be a finite labeled graph with all even labels that does not contain a 4-clique. Then for the associated Artin group $G = G_{\Gamma}$ we have  that   $G \otimes G$ and $G \wedge G$ are residually finite.

b) Let $G$ be  the fundamental group $G$ of an algebraically clean graph of free groups  i.e. edge groups are free factors in the vertex groups. Then  $G$ is $FP_{\infty}$ and  $G \wedge G$ is residually finite. 
\end{prop}

\begin{proof}
a) 
  By \cite[Thm. 1.5]{J-S} $G$ is residually finite and (cohomologically) good. By Corollary \ref{equi2}  $G \wedge G$ is residually finite.
Since $G^{ab}$ is torsion-free  $G \otimes G$ is residually finite if $G \wedge G$ is residually finite.

b) By \cite[Thm 1.1, Thm 1.2]{J-S} the fundamental group $G$ of an algebraically clean graph of free groups is
  (cohomologically) good  and residually finite.  
   By  Corollary \ref{equi2}  $G \wedge G$ is residually finite.
  
  Note that if $G$ is the fundamental group of a graph of groups with vertex groups all of  type $FP_m$ and edge groups of type $FP_{m-1}$ then $G$ is of type $FP_m$ \cite{Brown-cr}.  In particular $G$ is $FP_{\infty}$.
  
\end{proof}

  The following fact is well known but as we could not find a reference we include a short proof for completeness.
 
 \begin{lemma} \label{retract10}  Let $G$ be a residually finite group, $H$ a subgroup of $G$ that is a retract. Then $H$ is closed in the profinite topology of $G$ i.e. $H$ is separable in $G$.
 \end{lemma} 
 
 \begin{proof}
 Let $\rho : G \to H$ be a retraction map i.e. an epimorphism such that the restriction of $\rho$ on $H$ is the identity map. Let  $g \in G \setminus H$, then $g = kh$, where $k \in Ker (\rho) \setminus \{ 1 \}, h \in H$ and since $G$ is residually finite  there is a normal subgroup $N$ of finite index in $G$ such that $k \not\in N$. Set $M = \rho^{-1}( N \cap H) \cap N$ a normal subgroup of finite index in $G$. We claim that $g \not\in MH$, it is equivalent to $k \not\in MH$. Suppose $k = mh_0$, $m \in M, h_0\in H$. Then $1 = \rho(k) = \rho(m) \rho(h_0) = \rho(m) h_0$, so $h_0 = \rho(m^{-1}) \in \rho(M) \leq N$ and $k = mh_0 \in N $, a contradiction. Note that the profinite closure $S$ of $H$ in $G$ is the intersection of  $N_0 H$, where $N_0$ runs through the  finite index normal subgroups $N_0$ in $G$, hence $g \not\in  S$. As this holds for all  $g \in G \setminus H$, we conclude that $S = H$.
 \end{proof}

\begin{prop} Let $G$ be an even Artin group of $FC$ type. Then $G$ is (cohomologically) good and $FP_{\infty}$. In particular $G \wedge G$ and $G \otimes G$ are residually finite. \end{prop}

\begin{proof}
We induct on the number of vertices in the graph $\Gamma$ that defines, so   $G = G_{\Gamma}$. If $\Gamma$ is not the full graph then there are two vertices $v_1$ and $v_2$ not connected by an edge  and $\Gamma = \Gamma_1 \cup \Gamma_2$, where $v_i$ is a vertex of $\Gamma_i$. Denote $\Gamma_0 = \Gamma_1 \cap \Gamma_2$. Then $G = G_1 *_{G_0} G_2$, where $G_i = G_{\Gamma_i}$.  We claim that this decomposition is efficient with respect to the profinite topology.
The group $G$ is residually finite by \cite{BG-MP-L}.
Since the Artin group is even $G_0$, $G_1$ and $G_2$ are retracts of $G$, hence the profinite topology of $G$  induces on $G_i$ the profinite topology of $G_i$. Each $G_i$ is
closed in the profinite topology by Lemma \ref{retract10}.

If $\Gamma$ is the full graph, then since $G$ is of $FC$ type then $G$ is spherical , hence a direct product of irreducible spherical Artin groups ( i.e. the corresponding Coxeter group is finite). By the classification of the irreducible spherical Artin groups ( see \cite{Co} for the classification of  finite Coxeter groups, \cite{Hu} for the corresponding Dynkin diagrams ) we see that the only even ones are $\Z$ or dihedral. But both are (cohomologically) good, hence $G$ is (cohomologically) good. 

By  \cite{BG-MP-L} $G$ is residually finite and  obviously $H_2(G, \mathbb{Z})$ is finitely generated.  Then by  Corollary \ref{equi2}   $G \wedge G$ is residually finite.  Finally note that $G^{ab}$ is torsion-free and by Theorem \ref{old1}, d) $G \otimes G$ is residually finite.
\end{proof}

By definition Bianchi groups are defined as $PSL(2, O_d)$,
where $O_d$ is the ring of integers in the imaginary quadratic number field $\mathbb{Q} ( \sqrt{-d})$, where $d \in \mathbb{Z}, d \geq 1$.

 \begin{prop}
 Let $G = PSL(2, O_d)$ be a Bianchi group. Then $G \wedge G$ is residually finite.  
  Furthermore 
 in the case of the singly cusped Bianchi groups if  $d = -3, -11, - 19,  -67, - 163$ then $G \otimes G$ is residually finite.
 \end{prop}
 
 \begin{proof}  Bianchi groups are known to be  (cohomologically) good  because of the hierarchy properties \cite{G-J-Z}.  By \cite{G-S} $G$ is finitely presented, hence $H_2(G, \Z)$ is finitely generated.  The existing hierarchy as in \cite{G-J-Z} shows Bianchi groups  are residually finite. Hence by Corollary \ref{equi2}  $G \wedge G$ is residually finite.  
 
  In \cite{Re}, \cite{Swan} are calculated the abelianizations of the singly cusped Bianchi groups ( i.e. $d =  -1, -2, -3, -7, -11, -19,  -43, -67, -163$)  and the calculations there show that in general the abelianization might or might not contain 2-torsion, this all depends on $d$. There is no 2-torsion in the abelianization of $G$ precisely when  $d = -3, -11, - 19,  -67, - 163$. 
 \end{proof}

 \begin{prop}
 Let $G$ be a limit group. Then $G \wedge G$ is residually finite.
 \end{prop}
  
  \begin{proof}
  Limit groups are finitely presented. Since limit groups are fully residually free, they are residually free and hence residually finite. By \cite{G-J-Z} they are (cohomologically) good. Then by  Corollary  \ref{equi2}  $G \wedge G$ is residually finite. 
  \end{proof}
  
  \begin{prop}
  Let $G$ be a Coxeter group. Then $G \wedge G$ is residually finite.
  \end{prop}
  
  \begin{proof}
 Coxeter groups are linear over $\mathbb{C}$, in particular they are  residually finite. Recently Hughes, M\"oller and  Verghese proved that Coxeter groups are (cohomologically) good \cite{H-M-V}. By definition Coxeter groups are finitely presented. Then we can apply Corollary \ref{equi2}.
  \end{proof}

  \section{$G \otimes G$ and $G \wedge G$ for $G =SL_n(\mathbb{Z})$ and $G = GL_n(\mathbb{Z})$} \label{SLnSection}

\medskip
  In this section we use results from algebraic $K$-theory  to study the residual finiteness of $G \wedge G$ and $G \otimes G$ for $G = SL_n(\mathbb{Z})$ and $G = GL_n(\mathbb{Z})$ for $n \geq 3$. Our main reference is the book \cite{Kbook}.
For the definitions of the Steinberg group $St(R)$, $St_n(R)$ and the relevant groups $E(R)$ and $E_n(R)$ we refere the reader to \cite{Kbook}.

Let $R$ be a ring. By definition, see \cite{Kbook},  there is an exact sequence of groups
 $$
 1 \to K_2(R) \to St(R) \to GL(R) \to K_1(R) \to 1$$  
The image of the map $St(R) \to GL(R)$ is precisely  $E(R)$.

\begin{theorem} (Kervaire, Steinberg) \cite[Thm. 5.5]{Kbook} The Steinberg group $St(R)$ is the universal
central extension of $E(R)$. Hence
$K_2(R) \simeq H_2(E(R); \Z)$.
\end{theorem}

By \cite[Remark 5.5.2]{Kbook} (Stability for $K_2$) for $ K_2(n,R) = Ker(St_n(R) \to E_n(R)) $ if  $R$ is a Noetherian ring
of dimension $d$  then the natural maps $K_2(n,R) \to K_2(R)$ are isomorphisms for all $n \geq d + 3$. 

  \begin{prop} \cite[Prop. 5.5.1]{Kbook} \label{K2} Let $R$ be a ring and  $n \geq 5$ be an integer. Then
the Steinberg group $ St_n(R)$ is the universal central extension of $E_n(R)$. In particular
$K_2(n, \Z) = Ker (St_n(\Z) \to SL_n(\Z)) \simeq H_2(SL_n(\Z), \Z)$.
  \end{prop}
 
\begin{prop} \cite[Exer. 5.13]{Kbook}  \label{K4} If p is an odd prime, by  \cite[Thm
III.5.11.1]{Kbook}
$ K_2(\Z/p^n) =  0$ ( in additive notation) and if
$n \geq 2$ we have $ K_2(\Z/2^n) $ is a cyclic group of order 2.
  \end{prop}
  
  \begin{theorem} \cite[Ch. 10]{Milnorbook}
  The group $K_2(\Z)$ is cyclic of order 2.
  \end{theorem}

As a corollary of the above results and using the remark about the stability of $K_2$ we have

\begin{cor} \label{K3}  For $n \geq 5$ 
$$H_2(SL_n(\Z), \Z) \simeq K_2(n,\Z) \simeq K_2(\Z) \simeq \Z / 2 \Z \hbox{ the cyclic group of order }2.$$
\end{cor}

The following is a famous result of Bass-Lazard-Serre.

\begin{theorem} \cite{B-L-S} For $n \geq 3$ every finite index subgroup $H$ of $SL_n(\Z)$ is a congruence subgroup i.e. there is some $s \geq 2$ such that $Ker (SL_n( \Z) \to SL_n( \Z / s \Z)) \leq H$. Hence $ \widehat{SL_n(\Z)} \simeq SL_n( \widehat{\Z}) \simeq \varprojlim_s SL_n(\Z/ s \Z)$. 
\end{theorem}

By the above result  for $n \geq 3$ the finite group quotients of $SL_n(\Z)$ have a very specific form. Note this does not hold for $SL_2(\Z)$, since it is virtually free.

\begin{theorem} \label{SLn}
Let $G = SL_n( \Z)$ for $n \geq 3$. Then   the map  $H_2(G, \Z) \to \varprojlim_s H_2(G_s, \Z)$  is an isomorphism. In particular, $G \otimes G  \simeq G \wedge G$ are residually finite.
\end{theorem}

\begin{proof}
Let $m_s$ be an integer divisible by the order of $G_s = SL_n( \Z / s \Z)$ where $s$ is divisible by 4. By the Universal Coefficients Theorem
 there is a short exact sequence
 $$ 0 \to H_2 (G_s , \Z)  \otimes_{\Z} (\z / m_s \Z) \to H_2(G_s, \Z/ m_s \Z) \to Tor_1^{\Z}(H_1(G_s, \Z), \Z / m_s \Z) \to 0$$  
Since $m_s H_2(G_s, \Z) = 0$ and $H_1(G_s, \Z) \simeq G_s^{ab} = 0$ we get  $  H_2(G_s, \Z/ m_s \Z) \simeq H_2(G_s, \Z)$. By \cite{Kallen} we have
$$  H_2(G_s, \Z/ m_s \Z) \simeq H_2(G_s, \Z) \simeq (\Z/ 2 \Z) \times (\Z / 2 \Z) \simeq H_2(G, \Z) \hbox{ for } n = 3 \hbox{ and } n = 4$$
Furthermore
$$  H_2(G_s, \Z/ m_s \Z) \simeq H_2(G_s, \Z) \simeq K_2(\Z/ s \Z, n) \simeq K_2(\Z/ s \Z)  \hbox{ for } n \geq 5$$
where the second isomorphism follows from  Proposition \ref{K2}  and the last  isomorphism is the stability of $K_2$. 

Suppose $s = 2^{z_1} \ldots p_i^{z_i}$, with $z_1 \geq 2$, is the decomposition as a product of prime integers, hence
$$\Z/ 2 \Z \simeq K_2(\Z/ 2^{z_1} \Z) \simeq K_2(\Z/ s \Z)\simeq H_2(G_s, \Z) \simeq H_2(G_s, \Z/ m_s \Z) $$
where the first equality is Proposition  \ref{K4}, the second follows from the Chinese reminder theorem and  Proposition \ref{K4}. Note that if $R_1, \ldots, R_i$ are rings ( with unity) then $K_2(R_1 \times \ldots \times R_i) \simeq K_2(R_1) \times \ldots \times K_2(R_i)$.

\medskip
{\bf Claim.} $ \mu_s : H_2(G, \Z) \to H_2(G_s, \Z)  \hbox{ is an  isomorphism for } n \geq 3.$  In particular, the map $ H_2(G, \Z) \to \varprojlim_s H_2(G_s, \Z)$ is an isomorphism and by Lemma \ref{old1}, a2) $G \wedge G$ is residually finite. 

\medskip
Proof.  Since $|H_2(G, \Z)| = | H_2(G_s, \Z)|$ it suffices to show $\mu_s$ is surjective.
Consider  the short exact sequence $ 1 \to N_s \to G \to G_s \to 1$ and the corresponding 5-term exact sequence in homology  
$$ H_2(G, \Z) \to_{\mu_s} H_2 (G_s, \Z)  \to N_s/ [G, N_s] \to G^{ab} \to G_s^{ab} \to 0$$ 
Note that  $G^{ab} = G/ G' $ is trivial.
Then $\mu_s$ is surjective is equivalent to 
 $N_s = [G,N_s]$. The last equality follows from   \cite[Thm. 4.1, Cor. 4.3]{B-M-S}.  This completes the proof of the claim.
 
 \medskip

 Note that by Malcev theorem \cite{Mal}, every finitely generated linear group (over a commutative ring) is residually finite,   $G = SL_n(\Z)$ is residually finite.  
 
Finally since $\Delta(G) = Ker ( G \otimes G \to G \wedge G)$, $\Delta(G)$ is a quotient of Whitehead quadratic  functor $\Gamma(G^{ab})$ and $G^{ab} = G/ G'$ is the trivial group
 we conclude that $G \otimes G \simeq G \wedge G$ is residually finite.
 \end{proof}

\begin{lemma} \label{LHSspectral} Let $P = G \rtimes A$, where $A \simeq \Z / 2 \Z$ and $G = G'$. Then the inclusion map $G \to P$ induces an isomorphism $H_0(A, H_2(G, \Z)) \simeq  H_2(P, \Z)$.
\end{lemma}

\begin{proof}
Consider the LHS spectral sequence $E^2_{p,q} = H_p(A, H_q (G, \Z))$ converging to $H_{p+q} (P, \Z)$. Note that $$E_{2,0}^2 = H_2(A, \Z) = 0, \ E_{p,1}^2 = H_p(A, H_1(G, \Z)) = H_p(A, G^{ab}) = 0$$ 
The convergence of the spectral sequence implies $H_2(P, \Z) \simeq E_{0,2}^{\infty}$.
Furthermore since $A$ is a retract  the map $H_m(P, \Z) \to H_m(A, \Z)$, induced by the projection $P \to A$,  splits and hence $H_m(A, \Z) \simeq E_{m,0}^2$ embeds in $H_m(P, \Z)$. This implies that $$E_{m,0}^2 = E_{m,0}^{\infty} \hbox{ for all } m \geq 1$$ In particular all differentials that start at $E_{m,0}^r$, $r \geq 2, m \geq 1$,  are the zero maps. Recall $E_{2,1}^2 = 0$.
 Comparing bidegrees of the differentials we obtain  that $E_{0,2}^2 = E_{0,2}^3 = E_{0,2}^4 = E_{0,2}^{\infty}$,  thus $H_2(P, \Z) \simeq E_{0,2}^2 = H_0(A, H_2(G, \Z))$.
\end{proof}

\begin{theorem}
Let $P = GL_n( \Z)$ for  $n \geq 3$. Then the map  $H_2(P, \Z) \to \varprojlim_U H_2( P/U, {\Z}
)$ is an isomorphism, where $\widehat{P} \simeq \varprojlim_U P/U$ with $P/ U$ finite groups. In particular $P \wedge P$ is residually finite.  Furthermore $P \otimes P$ is residually finite.
\end{theorem}

\begin{proof}
Note that $P = G \rtimes A$, where $G = SL_n(\Z)$ and $A$ is a cyclic group of order 2 generated by the matrix $diag(-1,1, \ldots, 1)$. By Lemma \ref{LHSspectral} $$H_2(P, \Z) \simeq H_0(A, H_2(G, \Z)) $$

 Recall that  $G_s = SL_n(\Z/ s \Z)$, hence $\widehat{P}  \simeq \varprojlim_s (G_s \rtimes A)$. Similar to the proof of  Lemma \ref{LHSspectral} $H_2(G_s \rtimes A, \Z) \simeq H_0(A, H_2(G_s, \Z))$ and  by the Claim in the proof of Theorem \ref{SLn}  the map $H_2(G, \Z) \to H_2(G_s, \Z)$, induced by the canonical projection $G \to G_s$, is an isomorphism, hence  
the map  $H_2(G \rtimes A, \Z) \to  H_2(G_s \rtimes A, \Z)$ is an isomorphism. Hence the map 
$$H_2(P, \Z) \to \varprojlim_s  H_2(G_s \rtimes A, \Z)$$ is an isomorphism.
Then  by Theorem \ref{old1}, a2)  $G \wedge G$ is residually finite.
Note that $P^{ab} \simeq A$  and for $x = diag(-1,1, \ldots, 1)$
we have $|x| = |x G'|$, hence by  \cite[Observation, p.87]{B-F-M} we have $P \otimes P \simeq \Delta(P) \times (P \wedge P)$ is residually finite. Alternatively we can use Theorem \ref{2tor} a) to deduce that $P \otimes P$ is residually finite since $P' = G$ has complement in $P$.

\end{proof}

\medskip
  
 \medskip
  \section{Constructible soluble groups, generalized  Baumslag-Solitar groups and others} \label{constr}

   The group $\X(G)$ was defined by Sidki in \cite{Said} by the following presentation
$$\X(G) = \langle G, {G}^{\varphi} | \ [g, {g}^{\varphi}] = 1 \hbox{ for } g \in G \rangle.$$ 
 Recall that $\nu(G)$ has a presentation
 $$\langle G, G^{\varphi} \ | \ [g_1, g_2^{\varphi}]^{g_3} = [g_1, g_2^{\varphi}]^{g_3^{\varphi}} = [g_1^{g_3}, (g_2^{ g_3})^{\varphi}] \rangle$$
By \cite[p.68-69]{Norai} \begin{equation} \label{isoiso}  \nu(G)/ \Delta(G) \simeq \mathcal{X} (G)/ R(G)\end{equation}
where  $\Delta(G)$ be the subgroup of $G \otimes G$ generated by $\{ g \otimes g ~ | ~ g \in G \}$ that can be identified with the subgroup of $\nu(G)$ generated by $\{ [g, {g}^{\varphi}] \   | \ g \in G \}$ and   $R(G) = [G, L(G), G^{\varphi}]$.

   The proof of the following result uses a version of the BNS-invariant of nilpotent-by-abelian groups. We will explain briefly the part of the theory that we will use.
  
  Let $Q$ be a finitely generated abelian group. The character sphere $S(Q) = Hom (Q, \mathbb{R}) \setminus \{ 0 \} / \sim$ where two homomorphisms $\chi_1, \chi_2 \in Hom (Q, \mathbb{R}) \setminus \{ 0 \}$ are equivalent, i.e. $\chi_1 \simeq \chi_2$, if one is obtained by the other by multiplication with a positive real number.  Let $V$ be a finitely generated $\Z Q$-module and following Bieri and Strebel we define
  $$\Sigma_{V}(Q) = \{ [\chi] \ | \ V \hbox{ is finitely generated as } \Z Q_{\chi}-\hbox{module} \}$$
  where $Q_{\chi} = \{ q \in Q \ | \ \chi(q) \geq 0 \}$. By definition $\Sigma_{V}^c(Q) = S(Q) \setminus \Sigma_V(Q)$.
  
  A soluble group is called constructible if it can be built
up from the trivial group using finite extensions and HNN extensions.  An HNN extension with a base group $B$, associated subgroups $B_1$ and $B_2$ and stable letter $t$ is given by a presentation $\langle B, t \ | \ \phi(b) = t^{-1} b t \hbox{ for }b \in B_1 \rangle$ where $\phi: B_1 \to B_2 $ is an isomorphism between the subgroups $B_1$ and $B_2$ of $B$. Since in our case the HNN extension should be a soluble group, it does not contain  free non-abelian subgroups. Hence the HNN extension should be ascending i.e. one of the associated subgroups should be equal to the base group. 
  
\begin{theorem} \label{Bi-Str}  \cite[Thm. 5.2]{B-S} Let $N \to G \to Q$ be a short exact sequence of groups with $Q$ abelian,
$N$ nilpotent and $G$ finitely generated. Set $V = N/N'$. Then $G$ is constructible if and only
if $\Sigma^c_V(Q)$ is contained in an open hemisphere of $S(Q)$. In particular $G$ is constructible if
and only if  $ G/N'$ is constructible.
  \end{theorem}
We observe that in the above theorem we view $V$ as a $\Z Q$-module  via the $Q$-action induced  by the conjugation action of $G$. 
  
\begin{theorem} \label{sol}
Let $G$ be a soluble group of type $FP_{\infty}$. Then for almost all prime integers  $p$ the groups $\X(G)$, $\nu(G)$, $G \wedge G$ are virtually residually $p$-finite.
\end{theorem}
  
  \begin{proof}
 By a corollary to  \cite[Thm. B]{Krop}  a group that is soluble of type $FP_{\infty}$ is constructible. By \cite[Thm. 5.2]{B-S} constructible soluble groups have a subgroup of finite index $M$ that is nilpotent-by-abelian and the complements of the first BNS (Bieri-Newmann-Strebel) invariant of $M$  in the character sphere $S(M)$ is inside an open hemisphere of $S(M)$,  see Theorem \ref{Bi-Str}.  The converse holds, if a soluble group has a subgroup of finite index $M$ with the above property then $M$ is costructible, hence of type $FP_{\infty}$. This was used in \cite{K-S} to show that if $G$ is soluble, $FP_{\infty}$ then $\X(G)$ is soluble, $FP_{\infty}$.

 Note that the above description using the BNS invariant passes to quotients, hence quotients of soluble, $FP_{\infty}$ groups are of type $FP_{\infty}$ (though for general groups this does not hold). Then $$\nu(G)/ \Delta(G) \simeq \X(G)/ R(G) \hbox{  is }FP_{\infty}.$$ Since $G$ is finitely generated and $\Delta(G)$ is a quotient of $\Gamma(G^{ab})$ we deduce that $\Delta(G)$ is finitely generated, abelian , hence $FP_{\infty}$. Thus  both $\nu(G)/ \Delta(G)$ and $\nu(G)$  are  soluble, $FP_{\infty}$, hence constructible.
 
  By \cite{Strebel} or \cite[Subsection 11.2]{Le-Ro}  a constructible soluble group is linear ( over $\mathbb{Q}$), hence by Malcev´s result is  residually finite.  By \cite{Plat}  if $G \leq GL_n(K)$, where $K$ is a field of zero characteristic then all but finitely many prime integers $p$ we have that $G$ is virtually residually $p$-finite. Hence $\X(G)$, $\nu(G)$  and $\nu(G)/ \Delta(G)$  are virtually residually $p$-finite for all but finitely many primes $p$.
 Since $G \otimes G \simeq [G, G^{\varphi}]$ is a subgroup of $\nu(G)$  and $G \wedge G$ is a subgroup of $\nu(G)/ \Delta(G)$ we obtain that $G \otimes G$ and $G \wedge G$ are  virtually residually $p$-finite for all but finitely many primes $p$.  
  \end{proof}
  
  Following \cite{L} we call a group $G$ highly residually finite if $G$ is residually finite and $G \in {\mathcal{A}}_2$.

  \begin{prop} \label{Highly-res-fin} 
  Let $G$ be a  soluble group of type $FP_{\infty}$. Then $G$ is highly residually finite. In particular $G \in \mathcal{A}_2$.
  \end{prop}
  
  \begin{proof}
  By \cite[Prop. 3.1]{Gr-He}  there is a 
  sequence of subgroups $M \leq K \leq G_0 \leq G$ such that $G_0$ has finite index in $G$, $G_0/ M$ is torsion-free abelian, $K/ M \simeq \Z$, $M$ is nilpotent torsion-free and $K$ is constructible. Furthermore the proof of \cite[Prop. 3.1]{Gr-He}, specially the fact proved at the end of the proof  (about a specific version of the complement of a Bieri-Strebel invariant) that $\Sigma^c_{M_1 = M^{ab}}$ is a unique point ( actually there is another possibility an empty set) implies that $K$ is an ascending HNN extension with a finitely generated nilpotent base group $M_0$ (the empty set case corresponds to $M_0 = M$ is finitely generated), where $M_0 \leq M$.
  
  By \cite[Cor. 4.17]{L}  an ascending HNN extension of a virtually polycyclic group is highly residually finite, in particular  since  $M_0$ is finitely generated nilpotent, hence polycyclic, we can deduce that  $K$ is highly residually finite. Furthermore by \cite[Cor. 2.11]{L} since $K$ is finitely generated  and highly residually finite, any extension of $K$ by a highly residually finite groups is also highly residually finite. In particular $G_0$ is highly residually finite. Since being highly residually finite is invariant under going up and down with finite index, we conclude that $G$ is highly residually finite.
  \end{proof} 
  
  Note that the above result gives another proof of the fact that if $G$ is a soluble group of type $FP_{\infty}$ then $G \wedge G$ is residually finite, since by Proposition \ref{Highly-res-fin} $G \in \mathcal{A}_2$ and we can apply Corollary \ref{equi2}.

  \medskip
  The classical Baumslag-Solitar groups  are given by the presentation $$BS(n,m) = \langle a,t \ | \ t a^n t^{-1} = a^m  \rangle$$ This definition
was recently generalised by Button \cite{But}, who considered the class of fundamental groups of finite graphs of groups
with vertex and edge groups isomorphic to $\mathbb{Z}^n$ and he calls such groups generalised Baumslag–Solitar groups of rank $n \geq 1$.

 \begin{cor} Let $G$ be a generalised Baumslag-Solitar group of rank $n \geq 1$ that is residually finite. Then   $G \in {\mathcal{A}_2}$ and $G \wedge G$ is residually finite.  In particular if $n = 1$ then $G \in \mathcal{A}_s$ for every $s \geq 2$. \end{cor}
 
 \begin{proof} By \cite{L-Sh} a generalised Baumslag-Solitar group of rank $n \geq 1$  that is residually finite is either virtually $\mathbb{Z}^n$-by-(finitely generated free) or an ascending HNN extension of $\mathbb{Z}^n$.
 In the former case we note that finitely generated free groups and $\mathbb{Z}^n$  are (cohomologically) good, hence by Lemma \ref{Serre2} b)  a group that is  $\mathbb{Z}^n$-by-(finitely generated free) is (cohomologically) good. Then by Lemma \ref{Serre2} a)  any finite extension of (cohomologically) good group is (cohomologically) good. Note that in this case we do not expect $\X(G)$ to be of type $FP_{\infty}$ as by \cite{B-K} for a non-cyclic finitely generated free group $F$ we have that $\X(F)$ is not $FP_3$. Since $G$ is (cohomologically) good and finitely presented,   by  Corollary \ref{equi2}  the map $H_2(G, \Z) \to H_2( \widehat{G}, \widehat{\Z})$ is injective   and   $G \wedge G$ is residually finite.
 
In the case  $G$ is an ascending HNN extension of $\mathbb{Z}^n$, $G$   is a constructible group and we can apply Theorem  \ref{sol}  to deduce that $G \in \mathcal{A}_2$.

Finally if $n = 1$ then $G$ has cohomological dimension $cd(G) \leq 2$, hence for every  subgroup $U$ of finite index in $G$ we have $cd(U) \leq cd(G) \leq 2$, so $H_s(U, - ) = 0 = H^s(U, -)$ for $s \geq 3$. It remains to consider only the case when the group is an ascending HNN extension over a vertex group i.e.  it is a classical Baumslag-Solitar metabelian  group $BS(1,m)$. Then its profinite completion is $\widehat{G} = (\prod_{p \nmid m} \Z_p) \rtimes \widehat{\Z}$ contains $ \widehat{U}$ as a subgroup of finite index  and  $H_s( \widehat{U}, - ) = 0 = H^s( \widehat{U}, -)$ for $s \geq 3$, because $\prod_{p \nmid m} \Z_p$ and  $\widehat{\Z}$ are  projective profinite groups, hence  the cohomological $p_0$-dimension $cd_{p_0} ( \widehat{U})  \leq cd_{p_0} ( \widehat{G}) \leq 1 + 1 = 2$ for any prime $p_0$. Hence by Theorem \ref{surprise}  $G \in \mathcal{A}_s$ for $s \geq 3$. 
  \end{proof}

\section{The non-abelian $q$-exterior square $G \wedge^q G$} \label{q-section}

 In the following result we state the classes of groups $G$ we studied in the previous section  for which $G \in \mathcal{A}_2$ and $G$ is residually finite. The fact that a RAAG is in $\mathcal{A}_2$ was proved in \cite{L}. The group $SL_2(\mathbb{Z})$ is virtually free of finite rank, hence is in the class $\mathcal{A}_2$.  By Corollary \ref{CorNew} we conclude that  $G \wedge^q G$ is residually finite.
 
\begin{cor} \label{q-square-res-finite} Let $G$ be a group  of  one of the following types:

a)  even Artin groups of FC type,

b)  even Artin group whose  underlying graph does not contain 4-clique,

c) a LERF Artin group,

 d) a graph product of residually finite, (cohomologically) good groups,
 
 e)  a right angled Artin group (RAAG),
 
 f) a limit group, 
 
 g) a soluble group of type $FP_{\infty}$,

 h) a generalised Baumslag-Solitar group of rank $n \geq 1$ that is residually finite,
 
 i) $SL_2(\mathbb{Z})$,
 
 j) a Coxeter group.

Then $G \in \mathcal{A}_2$. In particular,  $G \wedge^q G$ is residually finite for $q \geq 2$.
\end{cor}

\begin{prop} \label{q-sln-prop}
Let  $q \geq 2$ be an integer and  $G = SL_n(\mathbb{Z})$ for $n \geq 3$. Then $G \wedge^q G$ is residually finite.
\end{prop}

\begin{proof}
By the Claim from the proof of Theorem \ref{SLn} the map
$H_2(G, \mathbb{Z}) \to H_2(G_s, \mathbb{Z}),$
induced by the canonical epimorphism $G \to {G}_s$, is an isomorphism.
Since $G$ and $G_s$ are perfect groups and using the universal coefficients theorem we get that
$$H_2(G, \mathbb{F}_q) \simeq H_2(G,\mathbb{Z}) \otimes  \mathbb{F}_q \hbox{ and } H_2(G_s, \mathbb{F}_q) \simeq H_2(G_s,\mathbb{Z}) \otimes  \mathbb{F}_q $$
Hence the map 
$H_2(G, \mathbb{F}_q) \to H_2(G_s, \mathbb{F}_q) \hbox{ is an isomorphism.}$
Then there is an isomorphism  $H_2( \widehat{G}, \mathbb{F}_q)$ $ \simeq \varprojlim_s  H_2(G_s, \mathbb{F}_q)$ and we conclude that the map 
$$H_2(G, \mathbb{F}_q) \to H_2(\widehat{G}, \mathbb{F}_q) \hbox{ is an isomorphism.}$$
Then by Corollary \ref{CorNew} $G \wedge^q G$ is residually finite.
\end{proof}

Remark: By Lemma \ref{Wh-Wh} if
 $G$ is a $q$-perfect group, i.e. $G = [G,G] G^q$,  then $G \otimes^q G \simeq G \wedge^q G$. In particular  for $G = SL_n(\mathbb{Z})$ we have that $G \otimes^q G \simeq G \wedge^q G$ is residually finite.

  \section{Coherence} \label{coherence}
 
  Let $W (G)$ be the kernel of the homomorphism $$\rho : \X(G) \to G \times G \times G$$ 
where   $g$ goes to $(g,g,1)$ and  $g^{\varphi}$ goes to $(1,g,g)$ for all $g \in G$.
 Let $W_0(G)$ be the kernel of the map $$\rho_0 : \nu(G) \to G \times G \times G$$ 
 where   $g$ goes to $(g,g,1)$ and  $g^{\varphi}$ goes to $(1,g,g)$ for all $g \in G$.
 Then $$Im (\rho_0) = \langle ( g_1, g_2, g_3) \ | g_1 g_2^{-1} g_3 \in G' \rangle  = Im (\rho)  $$
 Recall that by   \cite[Lemma 2.2]{Norai2} $$W(G)/ R(G) \simeq H_2(G, \mathbb{Z})$$
 where  $R(G) = [G, L(G), G^{\varphi}]$.
By \cite[p.68-69]{Norai} \begin{equation} \label{isoiso}  \nu(G)/ \Delta(G) \simeq \mathcal{X} (G)/ R(G)\end{equation} with $W_0(G)/ \Delta(G) \simeq W(G)/ R(G)$, hence
$W_0(G)/ \Delta(G) \simeq H_2(G, \mathbb{Z}).$
By definition $$D(G) = [G, {G}^{\varphi}] \hbox{ and } L(G) = \langle {g}^{\varphi} g^{ -1} \ | \ g  \in G \rangle$$ are normal subgroups of $\X(G)$ and $W(G) = D(G) \cap L(G).$
Furthermore by \cite{Said}  $[D(G), L(G)] = 1$, hence $W(G)$ is abelian.

  By definition a group is coherent if every finitely generated subgroup is finitely presented. Obviously a finitely generated free group is coherent, since every subgroup of a free group is free.
    
  \begin{prop}
  Let $F$ be a finitely generated free group of rank $n \geq 2$. Then 
  
  a) $\nu(F)$ is not coherent.
  
  b) $\X(F)$ is not coherent.
  \end{prop}
  
  \begin{proof}
a)  Let $W_0 = W_0(F), \Delta = \Delta(F)$. We have $$W_0/ \Delta \simeq H_2(F, \mathbb{Z}) = 0$$ and $\Delta = \langle g \otimes g \ | \ g \in F \rangle$ is finitely generated since $F$ is finitely generated. Thus $W_0$ is a finitely generated central subgroup of $\nu(F)$, hence is finitely presented. Then a subgroup $H$ of $\nu(F)$  that contains $W_0$ is finitely presented if and only if $H/ W_0$ is finitely presented.
  Thus to prove that $\nu(F)$ is not coherent it suffices to show that   $Im (\rho_0)$  is not coherent.
  
  Consider the subgroup $$S = \{ ( g_1, g_2, 1) \ | \ g_1 g_2^{-1} \in F' \} \leq Im (\rho_0) $$ We view $S$ as a subgroup of $F \times F$ by forgeting the last coordinate that is 1. In \cite{B-R} Baumslag and Roseblade showed that a finitely presented subgroup of $F \times F$ is either a free group or virtually is a direct product of two free groups, which  does not hold for $S$, so $S$ is not finitely presented. It is easy to see that $S$ is finitely generated. One way to prove this is to note that $S \simeq F' \rtimes K$, where $K \simeq F$ with the $K$-action  on $F'$ via conjugation in $S$ equal to the  $F$-action on its commutator subgroup via conjugation in $F$. Then  a finite generating set of $F'$ as a normal subgroup of $F$ together with a finite generating set of $K$ gives a finite generating set of $S$.
  
  Alternatively by \cite[Thm A]{B-R}  if $S$ is of type $FP_2$  over $\mathbb{Q}$ ( i.e. the trivial $\mathbb{Q} S$-module $\mathbb{Q}$ has a projective resolution where each projective module in dimension $\leq 2$ is finitely generated) then
  $S$ should have a finite index in $F \times F$, a contradiction. 
  
  b)
   In the case $n \geq 3$ it is known (but not published) that  $W = W(F)$ is not finitely generated, so we cannot proceed as in case a). But we can consider the same group $S  \leq Im (\rho)$ and its preimage $H$ in  $\X(F)$.
Note that  $Im (\rho)$ is a subdirect product in $F \times F \times F$ that maps surjectively on pairs, hence by the  Virtual Surjection on Pairs (VSP) Theorem   of \cite{BHMS} $Im (\rho)$ is finitely presented. Thus $W$ is finitely generated as $\X(F)$-module (via conjugation) but the action of $LD$ is trivial,  where $L = L(F)$, $D = D(F)$, so $W$ is actually finitely generated $Q$-module for $$Q = F / F' \simeq \X(F)  / DL$$ Note that $S$ maps surjectively to $Q$,  this follows from $\rho(F) \subseteq S$. Hence $W$ is a finitely generated $H$-module via conjugation. By construction $H/ W \simeq S$ is finitely generated, hence $H$ is finitely generated.
   
   Finally if $H$ is finitely presented then since $W$ is a finitely generated $H$-module we conclude that $S$ is finitely presented, a contradiction.
  \end{proof}


\begin{thebibliography}{99}

 \bibitem{A-L}	K. Almeida, I. Lima, Subgroup separability of Artin groups,
J. Algebra 583 (2021), 25 - 37
	
	\bibitem{A-B-K} 
 A. Anjos Teixeira, L. Barroso Rocha, Dessislava H. Kochloukova,On the residual finiteness of the non-abelian tensor square $G \otimes G$, the non-abelian exterior square $G \wedge G$ and the weak commutativity construction $\X(G)$,  	arXiv:2609.34057

\bibitem{Artin} E. Artin, Theory of braids, Ann. Math. 48 (1947) 101 - 126
	
	
 \bibitem{B-L-S}	H. Bass, M. Lazard, J.-P. Serre, Sous-groupes d’indice fini dans SL(n, Z), Bull. Amer.
Math. Soc. 70 (1964), 385–392
	
	


\bibitem{B-M-S} Bass, H.; Milnor, J.; Serre, J.-P., Solution of the congruence subgroup problem for $SL_n(n \geq 3)$ and $Sp_{2n}(n \geq 2)$,
Inst. Hautes Études Sci. Publ. Math. No. 33 (1967), 59 - 137

 \bibitem{B-S}	R. Bieri, R. Strebel, A geometric invariant for nilpotent-by-abelian groups, J. Pure
Appl. Algebra 25 (1982), 1-20 

\bibitem{B-R} G. Baumslag, J. E. Roseblade. Subgroups of direct products of free groups. J. London Math. Soc. (2) 30 (1984), 44–52

	\bibitem{Bieribook}
	{ R. Bieri.}
	{ Homological Dimension of Discrete Groups.}
	2ª ed.. London: Queen Mary College Mathematics Notes (Queen Mary College), 1981 
	
	
	
 \bibitem{BG-MP-L} R. Blasco-García, C. Martínez-Pérez, L. Paris, Poly-freeness of even Artin groups of FC type,
Groups Geom. Dyn. 13 (2019), no. 1, 309 - 32 


\bibitem{B-E} J Boler, B Evans, The free product of residually finite groups amalgamated along
retracts is residually finite, Proc. Amer. Math. Soc. 37 (1973), 50–52

\bibitem{B-F-M} 
R. D. Blyth, F. Fumagalli, M. Morigi,
Some structural results on the non-abelian tensor square of
groups, J. Group Theory 13 (2010), 83 - 94

\bibitem{BHMS} M. R. Bridson, J. Howie, C. F. Miller III, H. Short, On the finite presentation of subdirect products and the nature of residually free groups, American Journal of Mathematics 135 (4) (2013), 891 - 933 

\bibitem{B-K} 
Martin R. Bridson, Dessislava H. Kochloukova,  Weak commutativity and finiteness properties of groups,
Bull. Lond. Math. Soc. 51 (2019), no. 1, 168–180

\bibitem{Brown-cr} K. S. Brown, Finiteness properties of groups,
J. Pure Appl. Algebra 44 (1987), no. 1 - 3, 45 - 75

\bibitem{B-L1}   R.  Brown, J.-L.  Loday,  Excision  homotopique  en  basse  dimension, C.R. Acad. Sci. Paris Sir.  I Math. 298 (1984), 353 - 356

\bibitem{B-L2}   R.  Brown, J.-L.  Loday,  Van  Kampen  theorems  for  diagrams  of  spaces, Topology 26 (1987), 311 - 335


\bibitem{B-J-R} R. Brown, D.L. Johnson, E.F. Robertson, Some computations of non-abelian tensor products of groups, J. Algebra, 111 (1987), 177-202

\bibitem{But} Jack Button, Generalised Baumslag-Solitar groups and hierarchically hyperbolic groups, arXiv:2208.12688, 2022

 \bibitem{C-Wy} W. Cohen, J. Wykowski, Cohomological separability of Baumslag–Solitar
groups and their generalisations, Journal of Algebra 688 (2026) 660 - 693 

\bibitem{C-R} D. Conduche,  C. Rodr\'iguez-Fern\'andez,  Non-abelian tensor and exterior products modulo q
and universal q-central relative extensions, J. Pure Appl. Algebra 78 (1992) 139-160

\bibitem{Co} H.S.M. Coxeter, Discrete groups generated by reflections, Ann.
Math, 35:588–621, 193

\bibitem{Deligne} Pierre Deligne, Extensions centrales non residuellement finies de groupes arithmetiques, C. R. Acad. Sci. Paris Ser. A-B, 287 (1978), 203 - 208

 \bibitem{D}   R.  K.  Dennis,  In  search  of  new  homology  functors  having a close  relationship  to  K-theory
preprint  (Cornell  Univ.,  1976)

\bibitem{Ellis}   G. Ellis, The  nonabelian  tensor  product  of  finite  groups  is finite,
J. Algebra  111 (1987),
203 - 205

\bibitem{Ellis2} G. J. Ellis, Tensor products and q-crossed modules, J. Lond. Math. Soc. (2) 51 (1995), no. 2, 243-258


\bibitem{E-RF} G. J. Ellis, C. Rodríguez-Fernández, 
An exterior product for the homology of groups with integral coefficients modulo p,
Cahiers Topologie Géom. Différentielle Catég. 30 (1989), no. 4, 339-343

\bibitem{Fu-Pi}L. Funar, W. Pitsch, The Schur multiplier of finite symplectic groups,
 Bull. Soc. Math. France 148 (2020), no. 3, 515-527

\bibitem{Go-Gu}  D.L. Goncalves, J. Guaschi, Journal of Pure and Applied Algebra 182 (2003) 33 – 64 


\bibitem{Green} E. Green, Graph products, Ph.D. thesis, Univ. of Warwick, 1991 


\bibitem{Gr-He}  J. R. J. Groves, S. M. Hermiller, Isoperimetric inequalities for soluble groups,
Geom. Dedicata 88 (2001), no. 1-3, 239–254

\bibitem{G-S}
F. Grunewald, J. Schwermer,   Subgroups of Bianchi groups and arithmetic quotients of hyperbolic 3-space
Trans. Amer. Math. Soc. 335 (1993), no. 1, 47–78

\bibitem{G-J-Z} F. Grunewald, A. Jaikin-Zapirain, P. A. Zalesskii, Cohomological goodness and the profinite
completion of Bianchi groups, Duke Math. J. 144 (2008), no. 1, 53 - 72

\bibitem{G-R-S} 
N. Gupta, N. Rocco, 
 S. Sidki, Diagonal embeddings of nilpotent groups, Illinois J. Math. 30 (1986), no. 2, 274 - 283 


 \bibitem{Hewitt} P. R. Hewitt,  Extensions of residually finite groups
J. Algebra 163 (1994), no. 3, 757–772 



\bibitem{Hu} James E. Humphreys, Introduction to Lie Algebras and Representation Theory, Springer-Verlag, New York, 1972

\bibitem{Kallen}
W. van der Kallen, The Schur multipliers of SL(3,Z) and SL(4,Z)
Math. Ann. 212 (1974/75), 47 - 49

{\label{red} \bibitem{Ka-Tu} Christian Kassel, Vladimir Turaev,  Braid Groups, Graduate Texts in Mathematics, vol. 247. Springer, 2008}

\bibitem{H-K} 
J. Hillman, D. Kochloukova, Pro-p completions of $PD_n$-groups
Israel J. Math. 265 (2025), no. 2, 801–831

 \bibitem{Hsu-Wise} T. Hsu, D. T.  Wise, On linear and residual properties of graph products,
Michigan Math. J. 46 (1999), no. 2, 251–259 

\bibitem{H-M-V} Sam Hughes, Philip M\"oller, Olga Varghese, Profinite properties of Coxeter groups, arXiv 2505.08701

\bibitem{J-S}  K. Jankiewicz, K. Schreve, Profinite properties of algebraically clean graphs of free groups,
J. Pure Appl. Algebra 229 (2025), no. 1, Paper No. 107775 

 \bibitem{K-S} 
D. H. Kochloukova, S. Sidki, 
On weak commutativity in groups,
J. Algebra 471 (2017), 319 - 347 

\bibitem{Le-Ro} J. C. Lennox, D. J. S. Robinson, The theory of infinite soluble groups,
Oxford Mathematical Monographs, Oxford University Press, Oxford
(2004)


\bibitem{Lev2} Levitt, Gilbert, Quotients and subgroups of Baumslag-Solitar groups,
J. Group Theory 18 (2015), no. 1, 1–43




\bibitem{Krop} P. Kropholler, On groups of type $FP_{\infty}$,
J. Pure Appl. Algebra 90 (1993), no. 1, 55–67 

 \bibitem{LimaOliveira} B. C. R. Lima,  R. N.   Oliveira, Weak commutativity between two isomorphic polycyclic groups, J. Group Theory 19 (2016), no. 2,  239 - 248 

\bibitem{L-Sh}  J. Lopez de Gamiz Zearra, S. Shepherd,  Separability properties of higher rank GBS groups,
Bull. Lond. Math. Soc. 57 (2025), no. 4, 1171 - 1194 

\bibitem{L} K. Lorensen, Groups with the same cohomology as their profinite completions. J. Algebra 320 (2008), no. 4, 1704–1722

\bibitem{L2} K. Lorensen, Groups with the same cohomology as their pro-p completions,
J. Pure Appl. Algebra 214 (2010), no. 1, 6–14.

 \bibitem{Mal} A.I. Malcev, On isomorphic matrix representations of infinite groups of matrices , Mat.
Sb. 8 (1940), 405–422, Amer. Math. Soc. Transl. (2) 45 (1965), 1 - 18 

\bibitem{Milnorbook} J. Milnor, Introduction to Algebraic K-theory, Annals of Math. Study 72,
Princeton Univ. Press, 1971

\bibitem{Plat} V. Platonov, A certain problem for finitely generated groups,
Dokl. Akad. Nauk BSSR 12 (1968), 492–494

\bibitem{Put} 
A. Putman, The Picard group of the moduli space of curves with level structures, Duke
Math. J. 161 (2012), 623–674

 \bibitem{Reid} A. Reid, Profinite properties of discrete groups,
London Math. Soc. Lecture Note Ser., 422
Cambridge University Press, Cambridge, 2015, 73–104

\bibitem{Norai2} 
 N. R. Rocco, On weak commutativity between finite p-groups, p odd. J. Algebra 76 (1982), no. 2, 471 - 488

\bibitem{Norai}  N. R. Rocco, On a construction related to the nonabelian tensor square of a group, Bol. Soc. Brasil. Mat. (N.S.) 22 (1991), no. 1, 63 - 79


\bibitem{Re} 
T. Reese,  Presentations for singly-cusped Bianchi groups,
Topology Appl. 328 (2023), Paper No. 108443, 20 pp.

\bibitem{R-Z} L. Ribes, P. Zalesskii,  Profinite groups,
Ergeb. Math. Grenzgeb. (3), 40,
Springer-Verlag, Berlin, 2010

\bibitem{S} H. Schneebeli, Group extensions whose profinite completion is exact, Arch. Math.
31 (1978) 244–253

\bibitem{Serre}  J.-P. Serre, Galois cohomology,  Springer-Verlag, Berlin, 1997

 \bibitem{Said}  S. Sidki, On weak permutability between groups. J. Algebra 63 (1980), no. 1, 186  - 225 

\bibitem{Strebel} R. Strebel, Finitely presented soluble groups, Group Theory, essays for
Philip Hall (ed. K. W. Gruenberg and J. E. Roseblade), Academic Press
1984


\bibitem{Swan}  R.G. Swan, Generators and Relations for certain Special Linear Groups, Advances in Math. 6 (1971), 1 – 77


 \bibitem{Kbook} Ch. A.  Weibel, The K-book,
Grad. Stud. Math., 145,
American Mathematical Society, Providence, RI, 2013

\bibitem{Wil} G. Wilkes, Profinite Properties of 3-Manifold Groups, PhD Thesis,  Oxford 2018 

 \bibitem{W-Y} X. Wu, S. Ye, Splittings and poly-freeness of triangle Artin groups, Trans. Amer. Math. Soc. 378 (2025), no. 7, 4707 - 4737 

\end{thebibliography}
\end{document}